\documentclass[a4paper,12pt]{amsart}

\usepackage{fullpage}
\usepackage[driverfallback=dvipdfm]{hyperref}
\usepackage[all]{xy}
\usepackage{tikz}
\usepackage{amsaddr}
\usepackage{amsthm, amssymb, amsmath}
\usepackage{color}
\usepackage{enumerate}

\newtheorem{thm}{Theorem}[section]
\newtheorem*{thm*}{Theorem}
\newtheorem{cor}[thm]{Corollary}
\newtheorem{lem}[thm]{Lemma}
\newtheorem{prop}[thm]{Proposition}

\theoremstyle{definition}
\newtheorem{defn}[thm]{Definition}

\newtheorem{nt}[thm]{Notation}
\newtheorem{conv}[thm]{Convention}
\newtheorem{rem}[thm]{Remark}

\newtheorem*{rem*}{Remark}

\newtheorem*{theoremaux}{Theorem \theoremauxnum}
\gdef\theoremauxnum{1}

\def\R{{\mathbb R}}           
\def\O{{\mathcal O}}          
\def\F{{\mathcal F}}          
\def\G{{\mathcal G}}          
\def\e{\varepsilon}           

\def\Aut{\operatorname{Aut}}
\def\Out{\operatorname{Out}}
\def\Lip{\operatorname{Lip}}

\def\PL{\operatorname{Str}}

\def\vol{\operatorname{vol}}
\def\Min{\operatorname{Min}}

\def\unf{\operatorname{unf}}

\def\wt{\widetilde}

\title[Minset is locally finite]{The minimally displaced set of
  an irreducible automorphism is locally finite}
\author{Stefano Francaviglia}
\address{Dipartimento di Matematica of the University of
Bologna}
\email{stefano.francaviglia@unibo.it}
\author{Armando Martino}
\address{Mathematical Sciences, University of Southampton }
\email{A.Martino@soton.ac.uk}

\author{Dionysios Syrigos}
\address{Mathematical Sciences, University of Southampton }
\email{d.Syrigos@soton.ac.uk}

\begin{document}


\begin{abstract}
We prove that the minimally displaced set of a relatively irreducible automorphism of a free
splitting, situated in a deformation space, is uniformly locally finite. The minimally displaced set coincides with the train track points for an irreducible automorphism.

We develop the theory in a general setting of deformation spaces of free products, 
having in mind the study of the action of reducible automorphisms of a free group on the simplicial
bordification of Outer Space. For instance, a reducible automorphism will have invariant free factors, act on the corresponding stratum of the bordification, and in that deformation space it may be irreducible (sometimes this is referred as 
relative irreducibility).

\end{abstract}
\maketitle
\tableofcontents

\section{Introduction}
\subsection*{Overview}

In this paper we study deformation spaces of marked metric graphs of groups.

Since its first appearance on the scene (\cite{cv}), the celebrated Culler-Vogtmann Outer Space
became a classical subject of research. It turned out to be a very useful tool for
understanding properties of automorphisms of free groups (see for instance~\cite{BestvinaHandel,MR748994, MR721773,MR879856,hatcher,MR1396778,MR2395795}).
A typical object in the Outer Space
of $F_n$ is a marked graph with fundamental group of rank $n$, and locally Euclidean
coordinates are defined by turning graphs into metric graph by an assignment of positive
edge-lengths. Outer Space is not compact and there are basically two ways of going to infinity:
making the marking diverge or collapsing a collection of sub-graphs of a given element $X$ of
Outer Space. The second operation has a local flavour and it is similar to the operation of
pinching a curve of a surface. Attaching these ``collapsed'' points, leads one to define the simplicial bordification of the deformation space. If one starts with Culler-Vogtmann space, the result is the free splitting complex, which is related to the free factor complex (see for instance~\cite{BF1,BR,BG,handelmosher,HMII,HH,KR}).

Collapsing comes naturally into play when one is analysing a reducible automorphism of $F_n$ induced
by simplicial map $f:X\to X$ which exhibits an invariant collection of sub-graphs  of $X$.

Once collapsed, $X$ is turned into a graph of groups corresponding to a free splitting
of $F_n$. On the other hand, the collapsed part is not necessarily connected. This phenomenon has
led researchers to investigate more general deformations spaces. Namely deformation spaces of
(not necessarily connected) graph of groups, possibly with marked points or ``hairs'' (see for
instance~\cite{GuirardelLevitt,FM13,FM18I,FM18II,MR3342683,CKV}).

One of the main tools used to study the action of automorphisms on deformation spaces is
the theory of Stalling folds (\cite{Sta}) and the so-called Lispchitz metric
(\cite{FM11,FM12,MR3342683}). In particular, given an automorphism $\phi$, one can study the
displacement function $\lambda_\phi$ defined as $\lambda_\phi(X)=\Lambda(X,\phi X)$ (here
$\Lambda$ denotes the maximal stretching factor from $X$ to $\phi X$, whose logarithm is the
asymmetric Lipschitz metric). Of particular interest is the set $\Min(\phi)$ of minimally
displaced points. When $\phi$ is irreducible, this coincides with the set of points supporting train-track maps
(\cite{FM18I}) and its structure is particularly useful for example in building algorithm for
decision problems. It is used in~\cite{FM18II} for a metric approach to the conjugacy
problem for irreducible  automorphism of free groups (solved originally in~\cite{MR1396778})
and the reducibility problem of free groups (solved originally in~\cite{K1,K2}).

\subsection*{Main results of the paper}
If one is interested in effective procedures, one of the main problem is that general
deformation spaces have a simplicial structure that is not locally finite. So if one starts
from a simplex and wishes to enumerate neighbouring simplices, there is no chance to make this
procedure effective.

\medskip

In this paper we prove that the minset $\Min(\phi)$ for irreducible automorphisms of exponential
growth is locally finite; namely given a simplex intersecting $\Min(\phi)$,
one can give a finite of its neighbours so that any simplex {\em not} in that list, does not
intersect $\Min(\phi)$. This is the content of our Theorem~\ref{locfin}. Moreover, it is also uniformly locally finite, Corollary~\ref{unif}.

\begin{thm*}[Theorems~\ref{locfin} and \ref{unif}]
	Let $G$ be a group equipped with a free splitting, $\G$. Let $\phi$ be an automorphism of $G$ which preserves the splitting and is irreducible with $\lambda(\phi) > 1$. Then $\Min(\phi)$ - also seen as the points which support train track maps for $\phi$ - is uniformly locally finite both as a subset of the deformation space $\O(\G)$ and its volume 1 subspace, $\O_1(\G)$. 
\end{thm*}

\begin{rem}
	We note that the number $\lambda (\phi)$ in the Theorem above is the (minimal) displacement of $\phi$ relative to the splitting $\G$. 
	
	For instance, if one takes a relative train track representative for an automorphism of $F_n$, then one gets a free splitting of $G=F_n$ by taking the largest invariant subgraph (the union of all the strata except the top one). The resulting automorphism is irreducible in the corresponding relative space, and the number $\lambda(\phi)$ is the Perron-Frobenius eigenvalue of the top stratum. 
\end{rem}

\medskip

The application we have in mind for this kind of result is an effective study of the minset for a
reducible automorphism.

In general this minset is empty, but starting with  a reducible automorphism $\phi$ of a free group, one can collapse an invariant free factor to obtain a new deformation space on which the automorphism acts. If $\phi$ is relatively irreducible in that space, then its minset is locally finite. Otherwise, one can keep collapsing free factors until it is relatively irreducible.

It is easy to see that the minset for a reducible automorphism is not locally finite in
general, however for any automorphism and any simplex with a given displacement, there are only
finitely many possible simple folds which produce simplices of strictly smaller displacement - Corollary~\ref{FoldingCandidateRegular1}. 

The idea behind Theorem~\ref{locfin} is the following. For a minimally displaced point $X$, it is
known that folding an illegal turn of an optimal map $f:X\to X$ representing $\phi$ produces a
path in $\Min(\phi)$, called folding path. (See for instance~\cite{FM13,FM18I}). But it is also
clear that there are legal turns that can be folded without exiting $\Min(\phi)$, for instance,
this may happen at illegal turns for $\phi^{-1}$. The strategy is to understand which legal turns can be folded, and we are able to produce a finite list such that if a turn
$\tau$ is not in that list, then by folding $\tau$ one exits the minimally displaced set. In our terminology, folding a  {\em critical} turn could allow one to remain in the minset, whereas folding a {\em  regular} turn forces one to leave it. One then understands arbitrary neighbouring simplices, by looking at which of them can be reached by a (uniformly bounded) number of critical folds - these are the only ones that may be minimally displaced. However, one complication is that it is possible that a critical fold could {\em increase} the displacement, and a subsequent critical fold decrease it so that one re-enters the minimally displaced set. Nevertheless, our result produces a finite list containing all neighbouring simplices that are minimally displaced. 

\begin{rem}
We have written the paper for deformation spaces of free splittings of $G$, namely connected graph of
groups with trivial edge-groups. However, every result
of the paper remains true for deformations spaces of non-connected graph of groups, as
developed  for instance in~\cite{FM18I,FM18II}. This is because connectedness plays no role in our proofs. (In those papers, non-connectedness was crucial since the main argument was an inductive one.) 
  Nonetheless, we decided to stick to the connected case for the benefit of the reader.  
\end{rem}

\subsection*{Structure of the paper}
We have decided to write the paper in a reverse order; we start immediately with the core of
the paper, postponing the section of general definitions to the end. This is because the
definitions and terminology we use are quite standard, and the reader used to the subject can
start reading directly.

\section{Preliminaries}
We recall some notation here and, we refer to Section~\ref{definitions} for more details.
\begin{conv} Deformations spaces - here, of free splittings, even though the concept exists more generally - can be viewed either as spaces of trees, or graphs. We adopt
  here  the graphs-viewpoint, but one can easily pass
  from one viewpoint to the other by taking universal covers and $G$-quotients (see below for more details).
  
  Throughout the whole paper, $\G$ will denote a fixed free splitting of the group $G$. 
  That is, we write $G= G_1* \ldots * G_k*F_n$, but this need not be the Grushko decomposition of $G$. In fact, in the examples we have in mind, $G$ is a free group, and the free factors $G_i$ correspond to a collection of invariant free factors under some automorphism of $G$.
  $\O(\G)$ will denote the deformation space of a free splitting of a group $G$.

  The typical object $X\in \O(\G)$ is 
therefore a marked metric graph of groups, with trivial edge groups,
and whose valence one or two vertices have non-trivial vertex group. (One can also think of $X$ as a $G$-tree with trivial edge stabilisers, where the vertex stabilisers are precisely the conjugates of the $G_i$. Elements of $G$ which fix some vertex are called elliptic, and the others are hyperbolic.) Note that every $X\in \O(\G)$ has the same elliptic elements (and this characterises the points in the space). 
For a vertex $v\in X$ we denote $G_v$ its vertex group. If $G_v$ is trivial, then $v$ is said to be {\em free}.

These spaces -  $ \O(\G)$ - naturally occur in the bordification of
classical Culler-Vogtmann Outer Space, on collapsing invariant subgraphs.

We denote by $\Aut(\G)$ the group of automorphisms of $G$ which preserve the splitting; that is, each $G_i$ in the splitting is sent to a conjugate of another (possibly the same) $G_j$. That is, $\Aut(\G)$ the group of automorphisms of $G$ which preserve the elliptic elements. Similarly, $\Out(\G) = \Aut(\G)/ \operatorname{Inn}(G)$.
\end{conv}

\subsection{Graph of Groups and $G$-trees}
We recall some basic notions of Bass-Serre theory. The main references for this section are \cite{Bass} and \cite{Serre}.
Given a graph $\Gamma$, we denote by $V(\Gamma)$ the set of vertices of $\Gamma$ and by $E(\Gamma)$ the set of (oriented) edges of $\Gamma$. If $e$ is an (oriented) edge, we denote by $\iota(e)$, the initial vertex of $e$ and by $\tau(e)$, the terminal vertex of $e$.

\begin{defn}[Graph of Groups]
A graph of groups $X$ consists of a connected graph $\Gamma$ together with groups $G_v$ for every vertex $v \in V(\Gamma)$ and edge groups $G_e = G_{\bar{e}}$, and monomorphisms $\alpha_e : G_e \to G_{\tau(e)} $ for every (oriented) edge $e \in E(\Gamma)$.
\end{defn}

In this paper, we work with free products of groups, which by Bass-Serre theory, arise as the fundamental groups of graph of groups with trivial edge groups. For the remainder of the section, we suppose that all the graph of groups have trivial edge groups. In this case, we can see a graph of groups $X$, as a pair which consists of a graph $\Gamma$ and a collection of groups $G_v$, one for each $v$ vertex of $\Gamma$; we simply write $X = (\Gamma,\{G_v\} _{v \in V(\Gamma)})$. We refer to $\Gamma$, as the topological space (or the graph) associated to $X$.

Let $X = (\Gamma,\{G_v\} _{v \in V(\Gamma)})$ be a graph of groups.
An \textit{edge- path} $P$ (of combinatorial length $k \geq 0$) in $X$ is a sequence of the form $ (g_1,e_1,g_2,e_2,\dots,g_k,e_k,g_{k+1})$ where the $e_i$'s are oriented edges of $\Gamma$ such that the terminal point of each $e_i$ is the
initial point of $e_{i+1}$, and each $g_i$ is a group element from the vertex group based at the initial point of $e_i$; in this case, we simply write $P = g_1e_1\dots g_k e_k g_{k+1}$. In the present work, we also use ``paths'' which are not edge-paths, as their endpoints may not be vertices. The previous definition is extended from edge-paths to \textit{paths}, by allowing the first and the last segments $e_1,e_k$ of $P$ to be partial edges (note that if $\iota(e_1)$ is not a vertex, then we set $g_1$=1 and, similarly, if $\tau(e_k)$ is not a vertex then we set $g_{k+1} = 1$).
A path is called \textit{loop}, if the endpoint of its last edge coincides with the initial point of its first edge.
We say that a path $P$, as above, is \textit{reduced}, if whenever $e_i = \bar{e}_{i+1}$, $i=1,\dots,k-1$, then $g_i$ is not the trivial group element. Furthermore, we say that a loop $P = g_1 e_1 g_2 e_2 \dots g_k e_k g_{k+1}$ is \textit{cyclically reduced} if it is reduced and if $e_k = \bar{e}_1$ then $g_{k+1} g_1 \neq 1$.

We can always represent the conjugacy class of a loop in the form $g_1 e_1 g_2 e_2 \dots g_ke_k$; in this case, being cyclically reduced means that whenever $e_i = \bar{e}_{i+1}$, then $g_i$ is not the trivial group element, with the subscripts taken modulo $k$.

We say that two paths are {\em equivalent} if one can be obtained from the other by a sequence of insertions or deletions of inverse-pairs. We then denote by $\pi(X)$, the set of equivalence classes of edge paths. 
Then $\pi(X)$ has the structure of a groupoid, whose operation is the concatenation of paths. Note that $\pi(X)$ is not a group in general, for if $p,q \in \pi(X)$, then the concatenation $p*q = pq$ is only defined exactly when the terminal vertex of $p$ is the initial vertex of $q$. In our setting, every path $p$ is equal (in $\pi(X)$) to a unique reduced path (with the same endpoints), which we denote by $[p]$.

(Alternatively, we can take $\pi(X)$ to be the set of reduced edge-paths, under the operation of concatenation followed by reduction.) 

If $a,b$ are two points of $\Gamma$ (not necessarily vertices), we can define the set $P[a,b]$, of paths in $X$, from $a$ to $b$. If $a,b$ are vertices, then $P[a,b]$ can be seen as a subset of $\pi(X)$. In the special, case where $a=b$, the set of loops based at $a$, $P[a,a]$, form a subgroup of $\pi(X)$ (as any two loops based at $a$, can be concatenated, inducing a loop based at $a$). 

\begin{defn}[Fundamental group]
Let $X = (\Gamma,\{G_v\} _{v \in V(\Gamma)})$ be a graph of groups. If $u$ is a vertex of $\Gamma$, then the fundamental group $\pi_1(X,u) = P[u,u]$ of $\Gamma$ is the set of loops in $\Gamma$ based at $u$ (as elements of $\pi(X)$). In our case, $$\pi_1(X,u) =  (\ast_{v \in V(\Gamma)} G_v) \ast \pi_1 (\Gamma,u).$$
\end{defn}

Note that the isomorphism type of $\pi_1(X,u)$ does not depend on the choice of $u$. If $u,w$ are vertices of $\Gamma$, then $\pi_1(X,u), \pi_1(X,w)$ are conjugate in $\pi(X)$ (and we simply write $\pi_1(X)$, when we ignore the base point).

We can now define the universal cover (or Bass-Serre tree) of a graph of groups.
\begin{defn}[Universal Cover]
Let $X = (\Gamma,\{G_v\} _{v \in V(\Gamma)})$ be a graph of groups. If $v$ is a vertex of $\Gamma$, then we define the tree $T = \widetilde{(X,v)} = \widetilde{X}$, as follows:
\begin{itemize}
	\item The set of vertices $V(T)$ are exactly the ``cosets'' $pG_w$, $w \in V(\Gamma)$, where $p$ is a reduced edge- path from $w$ to $v$.
	\item There is an (oriented) edge from $p_1G_{w_1}$ to $p_2G_{w_2}$, if and only if there are elements $g_i \in G_{w_i}, i=1,2$ and an edge $e \in E(\Gamma)$, so that $p_2$ = $p_1g_{1}eg_2$.  
\end{itemize}
\end{defn} 

\begin{rem}
\begin{enumerate}[{i)}] \
	
	\item As every element of $G =\pi_1(X,v)$ is a loop based at $v$, there is natural simplicial (left) action of the fundamental group on the universal cover $T$.
	
	\item If $w$ is a vertex of $\Gamma$, then there is a unique (reduced) path $p_w$ of $\Gamma_V$, from $v$ to $w$. Then the stabiliser $Stab_T(pG_w) = pG_wp^{-1}$
	
	\item The stabilisers of edges are, by construction, trivial.
	
	\item One can define arbitrary points of $T$ in exactly the same way, by allowing the paths $p$ to be general paths (not just edge-paths) from $v$ to $w$, with the convention that $G_w$ is trivial if $w$ is not a vertex. 
\end{enumerate}
\end{rem}

Given a graph of groups with trivial edge groups, we described the action of its fundamental group on a tree with trivial edge stabilisers (and vertex stabilisers which coincide with the set of conjugacy classes of the vertex groups).

Conversely, if a group $G$ acts on a tree with trivial edge stabilisers, we can construct a graph of groups with fundamental group $G$.
\begin{defn}[Quotient graph of groups]
Let's suppose that $G$ acts on a tree $T$ with trivial edge stabilisers. Then we can define the quotient graph of groups $X$, by taking the quotient graph $\Gamma = G/T$ and as vertex groups the stabilisers of some orbit of each vertex (and trivial edge groups). In this case, $\pi_1(X)$ is (isomorphic to) $G$.
\end{defn} 
 
\begin{rem}
By Bass-Serre theory, we know that these two constructions (graph of groups and the universal cover) are equivalent and we can go from the level of the tree to the level of the graph of groups and vice versa. (We have done some choices of fundamental domains which are inessential, as they produce isomorphic structures, every time.)
\end{rem}

Now we restrict to graph of groups in the same relative outer space, as in our Convention. More specifically, we assume that our graph of groups, will be $\G$-graph of groups (see Definition \ref{G-graphs}, for more details), for the fixed splitting $\G$ of our group $G$.

As in our description for the $\G$-trees (see \ref{definitions}), every $\G$-graph as a point (simplex) of $\O(\G)$, is equipped with a marking.

\begin{defn}
A $\G$-graph dual to a $\G$-tree. Namely, it is a finite connected $G$-graph of groups $X$, along with an isomorphism, $\Psi_X: G \to \pi_1(X)$ - a marking - such that:
\begin{itemize}
	\item $X$ has trivial edge-groups;
	\item the fundamental group of $X$ as a topological space is $F_n$;
	\item the splitting given by the vertex groups is equivalent, via $\Psi_X$, to $\G$. That is, $\Psi_X$ restricts to a bijection from the conjugacy classes of the $G_i$ to the vertex groups of $X$.
\end{itemize}
\end{defn}

%
%

The following definition is not the definition of the morphism that is given by Bass in \cite{Bass}, as he requires the map to be a graph morphism (sending edges to edges and vertices to vertices). However, in our case, his method can be easily adjusted for the more general maps that we define.

\begin{defn}[Maps Between Graphs of Groups]

Let $X = (\Gamma, \{G_v\}), Y = (\Gamma ', \{H_w\})$ be two (marked) $\G$-graphs.

A map $F $ between $X,Y$ consists of:

\begin{enumerate}
	\item  Two maps $f_V,f_E$:
\begin{enumerate}
	\item $f_V : V(\Gamma) \to \Gamma'$ and
	\item  For each edge $e$ with $\iota(e) = v, \tau(e) = v' $, $f_E(e)$ is a path from $f_V(v)$ to $f_V(v')$.
\end{enumerate}
\item  Isomorphisms between the vertex groups:
$\phi_v : G_v \to H_{f_V(v)}$ (by abusing the notation, we set $H_{f_V(v)} = 1$, if $f_V(v)$ is not a vertex).

\end{enumerate}

\end{defn}

Note that after sub-dividing some edges of the co-domain, we can suppose that our maps are simplicial. For any two $\G$-graphs $X,Y$, a map $F: X \to Y$, induces a (natural) homomorphism $\Phi_F: \pi(X) \to \pi(Y)$ between the set of paths, which restricts to a homomorphism $\Phi_F: \pi_1(X) \to \pi_1(Y)$ between the fundamental groups (well defined up to composition with inner automorphisms). The maps that we define in the following definition, play the role, at level of $\G$-graphs, of $G$-equivariant maps between $\G$-trees.

\begin{defn}
We say that a map $F: X \to Y$ is a $\G$-map, if the induced homomorphism $\Phi_F$ is the ``change of markings'', i.e. $\Phi_F \Psi_X  = \Psi_Y $ (up to composition with inner automorphisms).
\end{defn}

\begin{rem}
\begin{enumerate}[i)] \ 
\item By Bass-Serre theory, it's clear now that the notion of $\G$-maps at the level of marked $\G$- graphs is equivalent to the notion of $G$-equivariant ``straight" maps at the level of $\G$-trees - that is, equivariant, surjective, continuous maps that are determined by the images of vertices by extending linearly over edges. In fact, we can move from one to the other by simply considering lifts and projections, using Bass-Serre theory (even if the lifts and the projections are not unique, as they depend on the choice of some fundamental domain, all of them are equivalent).
\item Note that we cannot talk about continuity between graph of groups, but we could alternatively consider the equivalent notion of graph of spaces (we simply change the vertex groups $G_i$ with a topological space $\Gamma_i$, with $\pi_1(\Gamma_i) = G_i$). In that case,  the maps between two graph of spaces would be continuous.
\end{enumerate}

\end{rem}


\subsection{Basic Notation}

\begin{nt} We will use the following standard notation:
  \begin{itemize}
  \item   $\O_1(\G)$ the volume one subspace of $\O(\G)$.
    \item $\Delta$ denotes an open simplex of $\O(\G)$.
  \item $\Delta_X$ denotes the simplex with underlying graph of groups $X$.
\item $\bar e$ denotes the inverse of  oriented edge $e$.  Same notation for paths.
\item $\gamma\cdot\eta$ denotes concatenation of paths.
\item $L_X(\gamma)$ denotes the reduced length in $X$ of a loop $\gamma$, if $X$ is seen as a
  $G$-graph.
  \item Folding a turn $\{a,b\}$ by an amount of $t$ means identify initial segments of $a$ and
  $b$ of length $t$.  This is always well-defined for small enough $t$.
\item Given an automorphism $\phi \in \Out(\G)$, $\lambda_\phi:\O(\G)\to\R$ denotes the displacement
  function $\lambda_\phi(X)=\Lambda(X,\phi X)$ (this is well defined as the inner automorphisms act trivially). For a simplex $\Delta$ we set
  $\lambda_\phi(\Delta)=\inf_{X\in\Delta}\lambda_\phi(X)$; we set
  $\lambda(\phi)=\inf_{X \in \O(\G)}\lambda_\phi(X)$.
\item An $\O$-map between elements of $\O(\G)$ is a map that realises the difference of
  markings. A straight map between elements of $\O(\G)$ is an $\O$-map with constant speed on
  edges. (See the definitions section on page~\pageref{definitions}).
\end{itemize}

\end{nt}

Since we decided to adopt the graphs-viewpoint, some words of explanation are needed about turns.
A turn at a non-free vertex $v$ of $X$
is given by the equivalence class of unoriented pair $\{g_1e_1,g_2e_2\}$, where $e_1,e_2$ are
 (germs of) oriented edges with the same initial vertex, $v$; $g_1,g_2$ are elements in the vertex-group $G_v$, and the equivalence
relation is given by the diagonal action of $G_v$. If we think $X$ as a $G$-tree, a turn at a non-free vertex $v$, is a $G_{v}$-orbit of an unoriented pair of edges $\{g_1e_1,g_2e_2\}$ where $g_1,g_2 \in G_{v}$ and $e_1, e_2$ are oriented edges of $X$ emanating at $v$. In other words, the projection of a turn in a tree is a turn in the quotient graph of groups, and, conversely, the lift of a turn in the graph of groups is a turn in the universal cover.
In any case, we denote the turn given by the class of
$\{a,b \}$ simply $[a,b]$. (Note that $[a,b]=[b,a]=[ga,gb]$)	
	
For the convenience of the reader, in this section we give the definitions for both points of view, by thinking $X$ both as $\G$-graph of groups and as $\G$-tree. More specifically, everything in this section is written, the graphs-viewpoint and we translate into the language of trees, if needed. In the rest of the paper, we use these two notions interchangeably, as everything can be easily translated between these two point of views.

\begin{defn}
  Let $X\in\O(\G)$. A turn $[x,x]$ is said {\bf trivial}.
  A turn $\tau=[a,gb]$ at a vertex $v$ of $X$ is called {\bf degenerate}, if $a=b$ (resp.,  if $a,b$ are in the same $G_v$-orbit); it is called non-degenerate otherwise. If $e$ is an edge starting {\em and} ending at $v$ (resp., if both endpoints of $e$ are in the same $G$-orbit, as $v$), it determines a non-degenerate turn at $v$. 
\end{defn}

\begin{defn}
Given a straight map $f: X \to Y$, at the graph level, we say that $f$ maps the turn $[a,gb]$ to the turn $[c,hd]$, if the initial paths of (combinatorial) length $1$ in the graph of groups (resp., if the initial edge or germ of edge), of $f(a)$ and $f(gb)$ are $c$ and $hd$ (in some order).

We will sometimes abuse notation and say that $f$ maps the turn $[a,gb]$ to the turn $[f(a), f(gb)]$, even though we really mean this to be the initial sub-paths of combinatorial length $1$ of the (in general) paths given.

We say that $[a,gb]$ is $f$-legal, if $f$ maps $[a,gb]$ to a non-trivial turn. If $f$ maps either of $a$ or $b$ to a vertex, then we say the turn is illegal. 

If, moreover, $X=Y$, then we say that $[a,gb]$ is 
$\langle\sim_{f^k}\rangle$-legal if $f^k$ maps $[a,gb]$ to a non-trivial turn for all integers $k \geq 1$. 
\end{defn}

\begin{lem}
  Let $X\in\O(\G)$ and $\tau=[a,gb]$ be a non-degenerate turn at a vertex $v$. Then
  (equivariantly) folding $a$ and $gb$ gives a new element in $\O(\G)$. 
\end{lem}
\begin{proof}
The proof is straightforward and left to the reader.
\end{proof}
\begin{rem}
  Note that if $e$ is an edge emanating from $v$, a non- free vertex of $X\in\O(\G)$, and if $g\in
  G_v$ is such that $<g>\neq G_v$, then by folding a degenerate turn $[e,ge]$ we obtain a tree with
  non-trivial 
  edge groups(resp., stabilisers). Namely, the new edge $e$ emanating from $v$ has edge group $<g>$ (resp., stabiliser).

  Therefore, in practice, a non-degenerate turn is the same as a ``foldable'' turn. 
\end{rem}

\begin{rem}
	\label{actions}
Suppose that $X \in\O(\G)$ and $f:X\to X$ is a straight $\O$-map. By this we mean that there are two $G$ markings on $X$ (resp., $G$-actions), and $f$ is $\G$-map (resp., $G$-equivariant) with respect to the two different markings (resp., actions) (both of which lie in the same deformation space, and hence have the same elliptic elements). 

We will describe the change of action in more detail (at graph level, we have a very similar description, by changing the marking, instead of changing the action).
Let's give these names; we will denote the first action by $\cdot$ and the second action by $\star$. Then we will always have that there is an element, $\phi \in \Aut(\G)$ such that, 

$$
f(g \cdot x) = g \star f(x) = \phi(g) \cdot f(x).
$$

Then, on iterating $f$, we get, 
$$
f^r(g \cdot x)= \phi^r(g) \cdot f^r(x). 
$$

\end{rem}

\medskip

Note that if $e$ is an edge emanating from a vertex $v$, then for every $\G$-map  $f$ (resp., $G$-equivariant map),
the degenerate turn $[e,ge]$ is $f$-legal for any $g\neq Id\in G_v$, as long as $f$ does not map $e$ to a vertex.

\begin{lem}\label{L0}
Let $X,Y\in\O(\G)$ and $f:X\to Y$ be a straight $\O$-map. Let $v$ be a vertex of $X$ and $G_v$ its
stabiliser. Let $\tau=[a,gb]$ be a turn at $v$, such that neither $f(a)$ nor $f(b)$ is a single vertex. If $\tau$ is $f$-illegal,  then
for any $g'\neq g\in G_v$ the turn $[a,g'b]$ is $f$-legal.
If $X=Y$ and $\tau$ is $\langle\sim_{f^k}\rangle$-illegal, then
for any $g'\neq g\in G_v$ the turn $[a,g'b]$ is $\langle\sim_{f^k}\rangle$-legal.
\end{lem}
\begin{proof} Let's prove the second claim first.
  Since $\tau$ is $\langle\sim_{f^k}\rangle$-illegal, then there is some power $r\geq1$ so that $f^r(a)$
  and $\phi^r(g) f^r(b)$ are the same germ - we are using the automorphism $\phi$ as in Remark~\ref{actions}. It follows that  $[f^r(a), \phi^r(g') f^r(b)]$ is degenerate and
  legal. Since $f$ is an $\O$-map, it follows that $f^{r+l}(\tau)=[f^{r+l}(a), \phi^{r+l}(g')f^{r+l}(b)]$ is degenerate
  and legal for any $l\geq 0$. Since $f$-images of illegal turns are $f^n$-illegal for any
  $n$, then $[f^m(a), \phi^m(g')f^m(b)]$ is legal also for $m\leq r$.

  First claim now follows by exactly the same argument with $r=1$ and $l=0$.
\end{proof}

\begin{defn}
  Let $X\in\O(\G)$ and let $\tau$ be a turn of $X$. For any loop $\gamma$ in $X$  (resp., a path with endpoints which lie in the same $G$-orbit)
  we denote by $$\#(\gamma,\tau)$$ the number of times that the {\bf cyclically reduced} representative of
  $\gamma$ crosses $\tau$. We recall that $\tau$ is not an oriented object, so we do
  not take in   account crossing directions.   
\end{defn}

The following lemma is almost tautological, but important for our purposes. 
\begin{lem}\label{LCR}
  Let $X,Y\in \O(G)$ and $f:X\to Y$ any straight map. If $\gamma$ is a $f$-legal path in $X$,
  then it is reduced. 
\end{lem}
\begin{proof}
  If $\gamma$ is not reduced, then it contains a sequence $\bar ee$, hence  a turn of the kind
  $[x,x]$. That turn cannot be $f$-legal. 
\end{proof}

\begin{defn}
  Let $X\in\O(\G)$, and $\tau$ be a turn of $X$. We say that $\tau$ is {\bf (non-) free} if it is
  based at a (non-) free vertex. We say that $\tau$ is {\bf infinite non-free} if it is based at a
  vertex with infinite vertex group (resp., infinite stabiliser).  We say that $\tau$ is {\bf finite non-free} if it is non-free
  and based at a vertex with finite vertex group (resp., finite stabiliser).
  Given an invariant sub-graph $Y\subseteq X$ (resp., $G$- sub-forest), we say that
    $\tau$ is in 
    $Y$ if both germs of edges of $\tau$ belong to $Y$.
\end{defn}

\section{Unfolding projections and local surgeries on paths}\label{s3}

Suppose $\Delta $ is a simplex in $\O(\G)$, with underlying graph of groups $X$.  
Let $\tau$ be a non-degenerate turn in $X$. We denote by $\Delta_\tau$ the simplex obtained by
(equivariantly) folding $\tau$. If $\tau$ is free and trivalent then $\Delta_\tau$ trivially
equals $\Delta$. Otherwise,  $\Delta$ is a codimension-one face of
$\Delta_\tau$. In the latter case there is a natural 
projection $\Delta_\tau\to\Delta$ corresponding to the collapse of the newly created
edge. Rather, we will use the {\bf unfolding projection}, which is defined as follows.

Given $Y\in\Delta_\tau$, we will define lengths of edges of $X$ so that
isometrically folding $\tau$ eventually produces $Y$.
Let $e_1,e_2$ be the edges defining $\tau$ (possibly $e_1=e_2$ is $\tau$ arises at an
edge-loop) and let $e$ be the extra edge added in $\Delta_\tau$ after folding $\tau$. 

Firstly, every edge of $X$, different to $e_1,e_2$, will have the same length as its length in
$Y$. Then, for $i=1,2$ we set the length of $e_i$ to be $L_{Y}(e_i)+L_{Y}(e)$ if $e_1\neq e_2$,
and $L_{Y}(e_i)+2L_{Y}(e)$ if $e_1=e_2$.
We denote the resulting metric graph, which is an element of $\Delta$, by $\unf_\tau(Y)$
and we say that it is obtained by {\bf unfolding $\tau$}. The map $$\unf_\tau:\Delta_\tau\to\Delta$$ 
is our unfolding projection.

\begin{lem}\label{unf}
  The map $\unf_\tau$ is surjective. Moreover,
  Folding $\tau$ by an amount
of $L_{Y}(e)$ produces a simplicial segment from $\unf_\tau(Y)$ to $Y$. 
\end{lem}
\begin{proof}
  The proof immediately follows from the construction.
\end{proof}

\medskip

We describe now local surgeries on paths. As above, let $\Delta$ be a simplex of $\O(\G)$
with underlying graph $X$. Since we  adopt  the graphs-viewpoint, then we may view $G$ as the fundamental group of the graph of groups given by $X$.



  If $ g_1 e_1 g_2 e_2 \dots g_k e_k$ is a loop in the graph of groups, then it crosses $k$ turns (including multiplicity); each sub-path of the form $e_i g_{i+1} e_{i+1}$ determines a turn,  $[\bar{e}_i,g_{i+1}e_{i+1}]$, where the indices are taken modulo $k$. (The specific metric on $X$ is not relevant for this discussion, merely the fact that we have a way of representing elements/conjugacy classes as loops in the underlying graph of groups for $X$.)

 Thus a path (or loop) is reduced (cyclically reduced) if the turns it crosses (cyclically crosses) are all non-trivial.

With this description, we may modify any given path by replacing one of the $g_i$ with some
other element $g$ in the same vertex group. The turns crossed by this new path are exactly the
same as the original, except for one turn $\tau=[\bar{e}_{i-1}, g_ie_i]$ which is replaced with
$[\bar{e}_{i-1}, ge_i]$. We denote the modified turn and loop respectively $$\tau_g\qquad\text{and}\qquad \gamma_{\tau,g}.$$

Putting everything in formulas we have: 
\begin{lem}[Turn-surgery of paths]\label{Surgery}
Let $\Delta$ be a simplex of $\O(\G)$, $\gamma=g_1e_1\dots g_ke_k$ a cyclically reduced loop
realised in 
the underlying graph of groups and $\tau=[\bar{e}_{i-1},g_ie_i]$ be a turn crossed by
$\gamma$. Let 
$v$ be the initial vertex of $e_i$. Then, if $\tau$ is non-degenerate, for any $g\neq g_i\in
G_v$ the loop $\gamma_{\tau,g}$  is cyclically reduced and satisfies
$$\#(\gamma_{\tau,g},\tau')=
\left\{\begin{array}{ll}
         \#(\gamma,\tau') & \text{if }\tau'\neq\tau,\tau_g\\
         \#(\gamma,\tau')-1 & \text{if }\tau'=\tau\\
         \#(\gamma,\tau')+1 & \text{if }\tau'=\tau_g
       \end{array}\right.$$
   \end{lem}
Moreover, if $\tau$ is degenerate (hence $e_{i-1}=\bar e_i$), than the same is true if in
addition we choose $g\neq id$.
   
   \begin{proof}
     Since $\gamma$ is cyclically reduced and $\tau$ is not degenerate, then $\gamma_{\tau,g}$ is
     reduced. The same holds true if $\tau$ is degenerate and $g\neq id$. The claim now easily
     follows by counting the number of times that a turn appears along $\gamma_{\tau,g}$.  
   \end{proof}
We introduce also a second surgery on paths. Let 
$\gamma=g_1e_1\dots g_ke_k$ denote a loop as above. Let $e=e_i$ be an oriented edge crossed by $\gamma$ at least twice and let $j$ be the
next index so that $e_j=e$. We can therefore form the loop  $g_je_i\dots g_{j-1}e_{j-1}$ (note that the formed loop starts with the group element $g_j$ instead of $g_i$, as in this case any turn which is crossed by the this new loop, was seen as a turn crossed by $\gamma$). 
We refer to such procedure as {\em edge-surgery}, and denote the resulting loop by $$\gamma_e.$$

Note that every turn (cyclically) crossed by $\gamma_e$ is also a turn crossed by $\gamma$, so if $\gamma$ is cyclically reduced, then  
$\gamma_e$ is cyclically reduced as well.
By construction, $\gamma_e$ crosses the oriented edge $e$ only once. Still, it may cross $\bar
e$ and other edges multiple times.   
   
   \begin{lem}[Edge-reduction of loops]\label{Surgery3}
Let $\Delta$ be a simplex of $\O(\G)$, $\gamma=g_1e_1\dots g_ke_k$ a cyclically reduced loop realised in
the underlying graph of groups. Then for every
$e_i$ there is a cyclically reduced loop $\gamma'$, obtained by recursive edge-surgeries on $\gamma$, such that
first, $\gamma'$ crosses $e_i$, and second,
$\gamma'$ crosses every oriented edge at most once. (Possibly $\gamma'=\gamma$ if $\gamma$ had those properties).
   \end{lem}
   \begin{proof}
     For a loop $\eta$ set $n(\eta)$ the total number of repetitions (counted with
     multiplicity) of oriented edges. So $\eta$ crosses any oriented edge at most once if and
     only if $n(\eta)=0$.   If $n(\gamma)>0$, then  there is $g_je_j\dots g_ie_i\dots g_le_l$
     a sub-path  of $\gamma$ containing $e_i$ and so that $e_j=e_l$ 
     (indices are taken cyclically). The loop $\gamma_{e_j}$ contains $e_i$ and
     $n(\gamma_{e_j})\leq n(\gamma)-1$. We conclude by arguing inductively as $n(\gamma)$ is strictly
     decreasing under edge-surgeries. 
   \end{proof}

There is a version of the previous lemma for turns:
\begin{lem}\label{Surgery4}
Let $\Delta$ be a simplex of $\O(\G)$, $\gamma=g_1e_1\dots g_ke_k$ a cyclically reduced loop realised in
the underlying graph of groups.
If $\tau = [e,g e']$ is a non-trivial turn, which is crossed by $\gamma$, then we can find some cyclically reduced loop $\gamma'$, obtained by recursive edge surgeries on $\gamma$, such that first, $\gamma'$ crosses $\tau$ and second, $\gamma'$ crosses every oriented edge at most once.
\end{lem}
\begin{proof}
Let $\gamma=g_1e_1\dots g_ke_k$ be a cyclically reduced loop, as above.  Without loss of generality, we can assume that $\gamma$ is of the form $\gamma = ge' \dots \bar{e}$ and there are no other occurrences of $e'$ or $\bar{e}$ in $\gamma$, as otherwise we can preform edge-surgeries to change $\gamma$ to a cyclically reduced loop satisfying this property and which still crosses $\tau$ (cyclically). 


Now suppose that there is some oriented edge $E$ which is crossed by $\gamma$ at least twice. In this case, if $e_i$ and $e_j$ are the first and the last occurrences of $E$ in $\gamma$, respectively, then we replace $\gamma$ with the cyclically reduced loop $\gamma_1 = ge' \dots g_{i-1}e_{i-1} g_j e_j \dots g_k e_k$ which still crosses $\tau$ and, in addition, crosses $E$ once. By arguing inductively on the number of repetitions, we can find a $\gamma'$ with the requested properties.
\end{proof}

\section{Critical and Regular turns}
Firstly we explain our strategy. 
Given $X\in\O(\G)$ which is
minimally displaced by an automorphism $\phi$, we want to control the number of ways we can fold
a turn of $X$, without exiting $\Min(\phi)$. If a straight map $f:X\to X$ representing $\phi$
sends an edge of a maximally stretched loop $\gamma$ across a turn $\tau$, then by folding $\tau$ we
decrease the length of $f(\gamma)$. ``Morally'', this is the only way we can decrease
stretching factors of loops, and if we fold a loop {\em not} in the image of an edge, we increase the displacement.
``Morally'' does not mean ``literally'', and in fact one has to (focus on legal loops in
tension graph, and) analyse what happens to the images of turns. Our plan is to select a finite
number of turns that will be enough to control the displacement.
These will be our ``simplex critical turns'' that we introduce at the end of this section. The upshot of this process will be that the folding of simplex regular (i.e. non-critical) turns, strictly increases
the displacement. We note that our set of critical turns won't be optimal, in the sense that
we may a priori increase the displacement also by folding a critical turn; for instance we include all free turns for convenience.

It would be
interesting to have a nice characterisation of exactly those turns whose folding do not
increase the displacement. The next lemma is the key observation we begin with. 

\begin{lem}\label{L2}
  Let $[\phi]\in\Out(\G)$. Let $\Delta$ be a simplex of $\O(\G)$ and $f$ be an optimal
  map representing $\phi$ on a point $X$ of $\Delta$. Let $\tau$ be a non-degenerate
  turn, and let $\Delta_\tau$ be the simplex obtained by
  folding $\tau$. Let $X^t$ denote the point of $\Delta_\tau$ obtained from $X$ by folding
  $\tau$ by an amount $t$.

  If there is an $f$-legal loop $\gamma$ in the tension graph of $f$ (see Definition~\ref{deftg}) such that 
\begin{equation}
  \label{in1}
      \#(f(\gamma),\tau)\leq\lambda_\phi(X)\#(\gamma,\tau)\qquad\text{(resp. with strict inequality)}\tag{$\heartsuit$}
\end{equation}
then $$\lambda_\phi(X^t)\geq\lambda_\phi(X) \qquad\text{(resp. with strict inequality)}.$$
\end{lem}
\begin{proof} For any legal loop $\gamma$ in the tension graph of $X$, in  $X^t$ we have   $$\lambda_\phi(X^t)=\sup_g\frac{L_{X^t}(f(g))}{L_{X^t}(g)}\geq\frac{L_{X^t}(f(\gamma))}{L_{X^t}(\gamma)}=\frac{\lambda_\phi(X)L_{X}(\gamma)-2t\#(f(\gamma),\tau)}{L_{X}(\gamma)-2t\#(\gamma,\tau)}.$$

  and $\lambda_\phi(X^t)\geq\lambda_\phi(X)$ is guaranteed (with strict inequality) provided that
$$\lambda_\phi(X)L_{X}(\gamma)-2t\#(f(\gamma),\tau)\geq
\lambda_\phi(X)(L_{X}(\gamma)-2t\#(\gamma,\tau))$$
(resp. with strict inequality), and that last inequality clearly reduces to (\ref{in1}).
\end{proof}

What we will do from now on is showing that, except for finitely many turns,  we can guarantee the
existence of a loop $\gamma$ satisfying the hypothesis of Lemma~\ref{L2}.

\bigskip

We now make a choice of a single non-trivial element $h_v\in G_v$, for each non-trivial
$G_v$. Some of our subsequent constructions will be dependent on this choice, but we will never
need to revise this choice so we will not need to refer to the specific elements.
We denote the collection of such chosen elements by $H$: $$H=\{h_v: v\text{ is a non-free vertex}\}.$$

\begin{defn}\label{ADelta}
        For any simplex $\Delta$ of $\O(\G)$ define a set of loops, $A_{\Delta}$ as
        follows: a cyclically reduced loop $g_1e_1\dots g_ke_k$ in the underlying
        graph of $\Delta$ is in 
        $A_\Delta$ if and only if
        \begin{enumerate}
        \item it crosses every (un-oriented) edge at most $4$ times, and
        \item every non-trivial $g_i$  belongs to $H$.
        \end{enumerate}
\end{defn}

\begin{rem}
	The reason for constructing $A_{\Delta}$ is that it is finite, and gives us a local coordinate system of loops which will be sufficient for calculating displacements and the Lipschitz metric, locally. 
\end{rem}

\begin{defn}\label{defncc}
	Let $[\phi]\in\Out(\G)$. For a simplex, $\Delta$, in the underlying graph of
        $\Delta$ we say that a turn $\tau$ is {\bf 
          candidate regular} if it is infinite non-free and $\#(\phi(\gamma),\tau)=0$ for all loops
        in $A_\Delta$;
        a turn is {\bf candidate critical} if it is not regular (so $\tau$ is critical if
        either its vertex group is finite or if it appears in $\phi(A_\Delta)$). We denote the set of
   candidate critical turns of $\Delta$ by $\mathcal{C}_C(\Delta)$. (We remark that even if we
   do not refer to  $\phi$ in the notation, the set $\mathcal C_C(\Delta)$ depends on $\phi$).  
\end{defn}

\begin{lem}\label{LB1}
  Let $X, Y \in \O(\G)$ and $f:X \to Y$ a straight map. Let $\Delta=\Delta_X$.
  Suppose $\xi$ is either an edge or a free turn of
  $X$, which is crossed by an $f$-legal loop $\gamma_0$. Then $\xi$ is also crossed by a $f$-legal
  loop $\gamma\in 
  A_\Delta$ which additionally crosses any oriented edge at most once. If $\gamma_0$ is
  in the tension graph, then so is $\gamma$. 
  
  Moreover, under the same hypotheses, if additionally $X=Y$ and $\gamma_0$ is $\langle\sim_{f^k}\rangle$-legal, then $\gamma$ may also be chosen to be $\langle\sim_{f^k}\rangle$-legal.
\end{lem}
\begin{proof}
  Let $\gamma_0$ be a legal loop crossing $\xi$. By Lemma~\ref{LCR} $\gamma_0$ is cyclically reduced.
By Lemmas~\ref{Surgery3} or \ref{Surgery4}, as appropriate, we can reduce $\gamma_0$, via edge-surgeries, to a loop $\gamma_1$,
  still crossing $\xi$, and which crosses
  any oriented edge at most once.
   In particular,
  $\gamma_1$ satisfies condition $(1)$ for belonging to $A_\Delta$.

  Since $\gamma_1$ is obtained from $\gamma_0$ by edge-surgeries, the turns (cyclically) crossed by $\gamma_1$ are also crossed by $\gamma_0$. Hence if $\gamma_0$ is $f$-legal (respectively $\langle\sim_{f^k}\rangle$-legal), then so is $\gamma_1$.

 We now perform turn surgeries on $\gamma_1$ to produce a loop in $A_{\Delta}$. Condition (1) of Definition~\ref{ADelta} is already satisfied, so we only need to concern ourselves with condition (2), which is about the non-free turns crossed by the loop. 
 
 Suppose that $\gamma$ contains a sub-path at a non-free vertex, v,  $e_{i-1} g_i e_i$, crossing the corresponding non-free turn, $[\bar{e}_{i-1}, g_i e_i]$. Let $h \in H$ be the corresponding group element of $G_v$. Then by Lemma~\ref{L0}, at least one of  $[\bar{e}_{i-1}, e_i]$ and $[\bar{e}_{i-1}, h e_i]$ is $f$-legal (respectively $\langle\sim_{f^k}\rangle$-legal). 
 
 Therefore, by making appropriate choices at each non-free turn crossed by $\gamma_1$, we can perform a sequence of turn surgeries to produce an $f$-legal  loop $\gamma$ (respectively $\langle\sim_{f^k}\rangle$-legal) which is in $A_{\Delta}$ and still crosses $\xi$ (since $\xi$ is unaffected by turn surgeries). 
 
  Moreover, $\gamma$ contains only edges that were originally
  edges of $\gamma_0$, so if
  $\gamma_0$ is in the tension graph, so is $\gamma$.
\end{proof}

\begin{rem*}
	Note that in the previous result, we prove that the path $\gamma$ crosses each oriented edge at most one, hence each un-oriented edge at most twice, even though the requirement for being in $A_{\Delta}$ is that it crosses each un-oriented edge at most 4 times. The reason is that we use two such loops in the following Lemma; we contruct the loops in Lemma~\ref{LB2} by using two loops from Lemma~\ref{LB1}. 
\end{rem*}

\begin{lem}\label{LB2}
  Let $X, Y \in \O(\G)$ and $f:X \to Y$ a straight map. Let $\Delta=\Delta_X$.
  Let $\tau=[a,gb]$ be a non-free $f$-legal turn so that both edges $a,b$ are crossed by $f$-legal loops. 
  Then, there exists a $f$-legal loop $\gamma$ which crosses $\tau$. If the loops for $a$ and $b$ are in
  the tension graph, then so is $\gamma$.

 Moreover, we could take $\gamma = \gamma' _{\tau, g_1}$, for some $g_1 \in G_v$ where $\gamma ' \in A_{\Delta}$.
  Finally, the same is true for
  $\langle\sim_{f^k}\rangle$-legality in the case where $X=Y$. 
\end{lem}
\begin{proof}
We orient $a,b$ so that $v$ is the common starting point. By Lemma~\ref{LB1} there exist legal
loops  $\gamma^a$ and $\gamma^b$, crossing $a$ and $b$ respectively, each
crossing any oriented edge at most once, so that $\overline{\gamma^a}$ and $\gamma^b$ and  are in $A_\Delta$; we choose, $\gamma_a, \gamma_b$ to start with $a$ and $b$ respectively.  The loop $\gamma=\overline{\gamma^a} g \gamma_b$ crosses
$\tau$ by construction, 
and it crosses any un-oriented edge at most $4$ times. Let $\omega$ be the non-free turn determined at
the concatenation  of the end $\gamma^b$ and the beginning $\overline{\gamma^a}$. By construction $\gamma$ is
legal except possibly at $\omega$. Hence, by Lemma~\ref{L0}, up to possibly replacing $\gamma$ with
$\gamma_{\omega,h_v}$ or $\gamma_{\omega,id}$ we may assume that $\gamma$ is legal. By Lemma~\ref{LCR} $\gamma$ is cyclically reduced.

Moreover, both $\gamma_{\tau,id},\gamma_{\tau,h_v}$ satisfy condition $(2)$ for
belonging to $A_\Delta$ and at least one of them is legal by Lemma~\ref{L0}.
Clearly if both $\gamma^a$ and $\gamma^b$ are in the tension graph of $f$, then so is $\gamma$.  
\end{proof}

\begin{rem}\label{RB3}
  If $f:X\to Y$ is a minimal optimal map, then any edge in the tension graph is crossed by a
  $f$-legal loop in the tension graph, and so satisfies
  hypothesis of Lemma~\ref{LB1}, and any non-free legal turn in the tension graph satisfies the
  hypothesis of Lemma~\ref{LB2}, in the tension graph.
  This is just by the definition of {\em minimal} optimal map
  (see Definition~\ref{defnmo}). Moreover, we recall also that if $\phi$ is irreducible and
  $f:X\to X$ is an optimal map representing $\phi$ on a minimally displaced point, then the
  tension graph of $f$ is the whole $X$ (see Lemma~\ref{wasfootnote}).
  \end{rem}

\begin{rem}
	Note that for Lemma \ref{LB2}, the hypothesis that the turns are non-free is essential, as the lemma fails for free turns.
	
	Let $\phi$ be the automorphism of $F_2 = <a,b>$, which sends $a$ to $aba$ and $b$ to $ba$. Then the iwip automorphism $\phi$ admits a natural train track representative - which we also call $\phi$ -  on the rose $R$, where we identify each petal of $R$ with an element of the free basis $\{a,b\}$.  Moreover, the turn $\tau = [a,b]$ is $\langle\sim_{\phi^k}\rangle$- legal, as for every positive integer $k$, $\phi^k(a),\phi^k(b)$ start with $a$,$b$, respectively.

	However, note that a legal loop cannot contain the cyclic subwords $a b^{-1}$ or $b a^{-1}$. Therefore the only legal loops are either positive or negative words in $a$ and $b$. In particular, the free turn $\tau$ is $\phi$-legal, but it cannot be extended to a $\phi$-legal loop.  
\end{rem}

\begin{lem}\label{LB4}
  Let $[\phi]\in\Out(\G)$ be an irreducible element, let $X\in\Min(\phi)$ and $f:X\to X$ a
  train track map representing $\phi$. Then, every edge of $X$ is crossed by a $\langle
   \sim_{f^k}\rangle$-legal loop. In particular, Lemma~\ref{LB1} holds true for any edge of
   $X$, and Lemma~\ref{LB2} for any non-free $\langle\sim_{f^k}\rangle$-legal turn.
 \end{lem}
 \begin{proof}
   Any train track map is also train track with respect to  $\langle
   \sim_{f^k}\rangle$-legality; namely, it maps  $\langle \sim_{f^k}\rangle$-legal paths to  $\langle
   \sim_{f^k}\rangle$-legal paths (\cite[Corollary~8.12]{FM13}).
   
   Since $\phi$ is irreducible and, the tension graph of $f$ is the whole of $X$
   and any vertex is at least two-gated with respect to  $\langle
   \sim_{f^k}\rangle$. Therefore, there exists a  $\langle \sim_{f^k}\rangle$-legal loop, $\gamma_0$, in $X$. The iterated images $f^n(\gamma_0)$ form a sub-graph of $X$ which is
   $f$-invariant. By irreducibility, that sub-graph must be the whole $X$. In particular any
   edge $e$ is in the loop $f^n(\gamma_0)$ for some $n$, and that loop is  $\langle
   \sim_{f^k}\rangle$-legal because $f$ is a train-track map. 
 \end{proof}

The following is just a list of immediate corollaries of previous lemmas. 
\begin{lem}\label{images}
  Let $[\phi]\in\Out(\G)$ and $\Delta$ a simplex in $\O(\G)$. Then $\mathcal C_C(\Delta)$ contains
  (at least) all turns of the following kinds:
	\begin{enumerate}[(i)]
        \item Free and finite non-free turns;
        \item $f$-images of finite non-free turns, where $f$ is any straight $\O$-map landing
          on $X$; 
        \item turns in the $f$-image of an edge crossed by some $f$-legal loop, where $f:X\to
          X$ is any straight map representing $\phi$. In particular those include:
          \begin{enumerate}[(a)]
          \item edges in the tension graph of $f$, when $f:X\to X$ is a minimal optimal map
            representing $\phi$;
          \item any edge, provided the tension graph of $f$ is the whole $X$, {\em e.g.} if
            $\phi$ is 
            irreducible and $f$ is an optimal map representing $\phi$ on the minimally
            displaced point $X$;
          \end{enumerate}

        \item turns in the $f$-image of a free turn crossed by some $f$-legal loop, where $f:X\to
          X$ is any straight map representing $\phi$.

	\end{enumerate}
\end{lem}
\begin{proof}
  $(i)$ is by definition.   For $(ii)$, note that, since $f$ is an $\O$-map, the $f$-image of a
  finite non-free vertex is again a finite non-free vertex.
    Cases $(iii)$ and $(iv)$ follow immediately from Lemmas~\ref{LB1} and~\ref{RB3}. In
    particular, case $(iii)-(a)$ follow from Lemma~\ref{LB1} by Remark~\ref{RB3}; case
    $(iii)-(b)$ from  Lemma~\ref{LB1} by Remark~\ref{RB3}; case $(iv)$ from Lemma~\ref{LB1}.
\end{proof}

\begin{prop}\label{CandLegal0}
  Let $[\phi]\in\Out(\G)$, $\Delta$ a simplex in $\O(\G)$,
  and $f:X\to X$ be a straight map representing $\phi$ at a point $X\in \Delta$.
  Let $\tau_1, \ldots, \tau_k$ be candidate-regular turns.
  Then for any $j=1,\dots,k$, any $f$-legal loop $\gamma_0$ crossing $\tau_j$ can be
  modified via turn-surgeries (at infinite non-free turns) to an $f$-legal loop $\gamma$ so that 
  \begin{enumerate}[(i)]
                \item $\#(\gamma,\tau_j)=1$ and,
		\item $\sum_{i=1}^k \#(\gamma,\tau_i) = 1$ and, 
		\item $\sum_{i=1}^k \#(f(\gamma),\tau_i) \leq 1$ and,
                \item  $\sum_{i=1}^k \#(f(\gamma),\tau_i) =0$ unless $f$ maps $\tau_j$ to some $\tau_l$.               
                \end{enumerate}
                Moreover, if in addition
                $\gamma_0$ is $\langle \sim_{f^k}\rangle$-legal, then $\gamma$ can be chosen to
                be $\langle \sim_{f^k}\rangle$-legal.  
              \end{prop}
      \begin{proof}
 Without loss of generality we may assume that $\gamma_0$ crosses $\tau_1$.
We modify $\gamma_0$ by using turn-surgeries (Lemma~\ref{Surgery}) in order to
get a new loop $\gamma$ that satisfies the extra properties. We will apply surgeries
only on turns at non-free vertices with infinite stabilisers, so there will be infinitely many
choices every time.

Concretely, let $\gamma_0$ be represented as $g_1e_1\dots g_ne_n$ (with cyclic indices modulo
$n$) so that $\tau_1=[\bar{e}_1,g_2e_2]$.
For any infinite non-free turn $\tau=[\bar{e}_i,g_{i+1}e_{i+1}]$ with $i\neq 1$, we choose an element
$a$ in the corresponding vertex group so that $\tau_a=[\bar{e}_i,ae_{i+1}]$ satisfies

\begin{enumerate}[(1)]
\item $\tau_a$ is not one of the $\tau_i$; 
\item $\tau_a$ is $\langle \sim_{f^k}\rangle$-legal;
\item $f(\tau_a)$ is not one of the $\tau_i$;
\end{enumerate}

Such an element exists because $\tau$ is infinite non-free, there are finitely many $\tau_i$, and
by Lemma~\ref{L0} all but one choice for $a$ produces a $\langle \sim_{f^k}\rangle$-legal turn. We define $\gamma$ as the
result of the turn-surgeries at all such infinite non-free vertices, by using the chosen
group elements.

Condition $(2)$ assures that $\gamma$ is legal, and $\langle \sim_{f^k}\rangle$-legal if
$\gamma_0$ where so. Since we did not touch $\tau_1$, condition
$(1)$ gives us point $(i)$ and $(ii)$. As for $(iii)$, let's analyse the turns crossed by
$f(\gamma)$. They come in several types: 
\begin{enumerate}[(a)]
	\item a turn crossed by the $f$-image of an edge of $\gamma$,
	\item the $f$-image of a free turn of $\gamma$,
	\item the $f$-image of a finite non-free turn of $\gamma$,
	\item the $f$-image of an infinite non-free turn of $\gamma$. 
\end{enumerate} 
By Lemma~\ref{images}, the first three are all candidate critical, so none of the $\tau_i$
appears in this way (note that in type (b), the free turns that appear, are crossed by the $f$-legal loop $\gamma$, so the hypothesis of \ref{images} (iv) is satisfied). Crossings of kind $(d)$ are avoided by condition $(3)$, except possibly
if $f(\tau_1)$ equals one of the $\tau_i$'s. Points $(iii)$ and $(iv)$ follow.

\end{proof}

\begin{cor}\label{CandLegal2}
  Let $[\phi]\in\Out(\G)$, $\Delta$ a simplex in $\O(\G)$,
  and $f:X\to X$ be a minimal optimal map representing $\phi$ on a point $X\in \Delta$. 
  
	Suppose that $\tau_1, \ldots, \tau_k$ are candidate regular turns. If there is a turn
        $\tau_j$ 
        which is $f$-legal and in the tension graph of $f$, then there exists an
        $f$-legal loop, $\gamma$, in the tension graph, and  such that:
	\begin{enumerate}[(i)]
        \item $\#(\gamma,\tau_j) = 1$ and, 
        \item $\sum_{i=1}^k \#(\gamma,\tau_i) = 1$ and,
        \item $\sum_{i=1}^k \#(f(\gamma),\tau_i) \leq 1$ and,
        \item $\sum_{i=1}^k \#(f(\gamma),\tau_i) =0$ unless $f$ maps $\tau_j$ to some $\tau_l$.   
	\end{enumerate}

      \end{cor}
      \begin{proof}
        By Remark~\ref{RB3}, Lemma~\ref{LB2} applies for the infinite non-free turn, $\tau_j$. So there is a 
$f$-legal loop  $\gamma_0$, in the tension graph, and crossing $\tau_j$.
Proposition~\ref{CandLegal0} applies. Since $\gamma$ is obtained from $\gamma_0$ via
turn-surgeries, and since $\gamma_0$ is in the tension graph, so also $\gamma$ is in the
tension graph.
\end{proof}

\begin{cor}\label{candlegal3}
  Let $[\phi]\in\Out(\G)$ an irreducible element, $\Delta$ a simplex in $\O(\G)$,
  and $f:X\to X$ a train-track map representing $\phi$ on a point $X\in \Delta$. 
  
	Suppose that $\tau_1, \ldots, \tau_k$ are candidate regular turns. If there is a turn
        $\tau_j$ which is $\langle \sim_{f^k}\rangle$-legal, then there exists a
        $\langle \sim_{f^k}\rangle$-legal loop, $\gamma$   such that:
	\begin{enumerate}[(i)]
        \item $\#(\gamma,\tau_j) = 1$ and, 
        \item $\sum_{i=1}^k \#(\gamma,\tau_i) = 1$ and,
        \item $\sum_{i=1}^k \#(f(\gamma),\tau_i) \leq 1$ and,
        \item $\sum_{i=1}^k \#(f(\gamma),\tau_i) =0$ unless $f$ maps $\tau_j$ to some $\tau_l$.   
          
	\end{enumerate}

\end{cor}
\begin{proof}
  Lemma~\ref{LB4} applied for $\tau_j$ (which is necessarily infinite non-free, as it is regular)  guarantees
the existence of a $\langle \sim_{f^k}\rangle$-legal
loop  $\gamma_0$  crossing $\tau_j$.
Proposition~\ref{CandLegal0} applies.
\end{proof}

\begin{cor}\label{C2}
  Let $[\phi]\in\Out(\G)$. Let $\Delta$ be a simplex of $\O(\G)$ and $f$ be a minimal
  optimal   map representing $\phi$ on a point $X$ of $\Delta$. Let $\tau$ be a non-degenerate
  candidate  regular turn, and let $\Delta_\tau$ be the simplex
  obtained by 
  folding $\tau$. If $X^t$ denotes the point of $\Delta_\tau$ obtained from $X$ by folding
  $\tau$ by an amount $t$, then $$\lambda_\phi(X^t)\geq\lambda_\phi(X).$$
  Moreover, if $\tau$ is $f$-legal and in the tension graph, and if $\lambda(\phi)>1$, then
  the inequality is strict. 
\end{cor}
\begin{proof}
  By Remark~\ref{RB3} and Lemma~\ref{LB1}, there exists an $f$-legal loop $\gamma \in A_{\Delta}$ in the
tension graph, and any turn which is crossed by the image of $f(\gamma) = \phi (\gamma)$ is
candidate critical just by definition of candidate critical.
Since $\tau$ is regular
$$\#(f(\gamma),\tau)=0,$$ the non-strict version
of hypothesis~(\ref{in1}) of Lemma~\ref{L2} is fulfilled, and first claim follows.

If in addition  $\tau$ is legal and in the tension graph, we invoke Corollary~\ref{CandLegal2} (with $k=1$) to build a legal loop
$\gamma$ in the tension graph so that $\#(\gamma,\tau)>0$ and
$$\#(f(\gamma),\tau)\leq \#(\gamma,\tau).$$ The non strict version of inequality~(\ref{in1})
follows because $\lambda_\phi(X)\geq 1$. Moreover, 
if $\lambda(\phi)>1$, then $\lambda_\phi(X)\geq\lambda(\phi)>1$ and also the strict version is proved.
\end{proof}

\begin{cor}\label{FoldingCandidateRegular1} Let $[\phi]\in \Out(\G)$. 
	Let $\tau$ be a non-degenerate candidate regular turn with respect to a simplex
        $\Delta$ of $\O(\G)$, and $\Delta_{\tau}$ be the simplex obtained by folding
        $\tau$.Then $$\lambda_\phi(\Delta_{\tau}) \geq \lambda_\phi(\Delta).$$
\end{cor}
\begin{proof}
Any point $Y\in\Delta_\tau$ is obtained by folding $\unf_\tau(Y)$ (Lemma~\ref{unf}).
Corollary~\ref{C2} tells us $\lambda_\phi(Y)\geq\lambda_\phi(\unf_\tau(Y))$. The claim follows
taking infima. 
\end{proof}

Corollary~\ref{FoldingCandidateRegular1} provides the kind of non-strict inequalities we are
searching for. We now focus on turns whose folding guarantees the strict inequality $\lambda_\phi(\Delta_{\tau}) > \lambda_\phi(\Delta)$.

\begin{lem}\label{L3}
	Let $[\phi]\in \Out(\G)$.  Let $X\in \O(\G)$ and $f:X\to X$ be an optimal map representing
	$\phi$. For any $X_0\in\Delta_X$ let $f^0:X_0\to X_0$ denote the map $f$ read in $X_0$.
	
	There is a neighbourhood $U$ of $X$ in $\Delta_X$ such that for any $X_0\in U$
	there is a minimal optimal map $f_0:X_0\to X_0$  such that
	$$d_\infty(f^0,f_0)<\e$$
	and, for any $x,y\in X$ $$|d_X(f(x),f(y))-d_{X_0}(f^0(x),f^0(y)|<\e.$$ 
\end{lem}
\begin{proof}
	The function $\lambda_\phi(X)$ is continuous on $X$ and, tautologically, the metric of
	$X$  changes continuously on $X$.  Therefore, for any $\e>0$ there is a
	neighbourhood $U$ of $X$ in $\Delta_X$ such that
	\begin{itemize}
		\item $|\lambda_\phi(X_0)-\lambda_\phi(X)|<\e$,
		\item $d_\infty(\PL(f^0),f^0)<\e$,
		\item $|\Lip(\PL(f^0))-\Lip(f)|<\e$.
	\end{itemize}   
	
	By~\cite[Theorem~3.15]{FM18I}  there exists a weakly optimal map $f_1:X_0\to X_0$ representing
	$\phi$ such that $$d_\infty(f_1,\PL(f^0))\leq \vol(X_0)(\Lip(\PL(f^0))-\lambda_\phi(X_0))$$
	and, by~\cite[Theorem~3.15 and Theorem~3.24]{FM18I} there exists a minimal optimal map
	$f_0:X_0\to X_0$ representing $\phi$ such that
	$$d_\infty(f_1,f_0)<2\e.$$
	Putting together all such inequalities, and since $\e$ is arbitrary, we get that for any $\e>0$
	there is $U$ so that for all $X_0\in U$ we have $$d_\infty(f^0,f_0)<\e.$$
	
	Moreover, it is clear that we can choose  $U$ in such a way that for any $x,y\in X$ we have $$|d_X(f(x),f(y))-d_{X_0}(f^0(x),f^0(y)|<\e.$$ 
\end{proof}

\begin{lem}\label{L4}
	Let $[\phi]\in \Out(\G)$.  For any $X\in \O(\G)$ and optimal map $f:X\to X$ representing
	$\phi$, there is a neighbourhood $U$ of $X$ in $\Delta_X$ such that for any $X_0\in U$ there is a minimal optimal map $f_0:X_0\to X_0$ such that if  $\tau$ is a non-free turn in $X$ which is $f$-legal, then $\tau$ is $f_0$-legal. 
\end{lem}
\begin{proof}
	We apply Lemma~\ref{L3}. By equivariance, if $v$ is the
	non-free vertex where $\tau$ is based,  then $f(v)=f_0(v)$ is a non-free vertex. 
        Estimates of Lemma~\ref{L3} now easily imply that if $\tau$ is legal in $X$, it remains legal for small perturbations.
\end{proof}

\begin{lem}\label{Cylinder} Let $[\phi]\in\Out(\G)$, $\Delta$ be a simplex in $\O(\G)$, 
  and  $X\in\Delta$ be a point which is minimally displaced by $\phi$.
Suppose that $\Delta '$ is a simplex with face $\Delta$ and that there is a point $Y \in \Delta'$ which is minimally displaced by $\phi$. Then for any open neighborhood $U$ of $X$ in $\Delta ' $ there is a point $Z$ in $U$ which is minimally displaced by $\phi$.
\end{lem}
\begin{proof}
This is an immediate application of the convexity properties of $\lambda_{\phi}$ (namely, by ~\cite[Lemma~6.2]{FM18I}). More specifically, for $X,Y$ as above, the linear segment $\overline{YX}$ eventually enters in $U$, by continuity. On the other hand, by convexity properties of $\lambda_\phi$, any point of the segment $\overline{YX}$ is minimally displaced by $\phi$, which gives us the required result.
\end{proof}

\begin{prop}\label{FoldingCandidateRegular2} Let $[\phi]\in\Out(\G)$ be irreducible and with
  $\lambda(\phi)>1$. 
  Let $\Delta$ be a simplex of $\O(\G)$, and let $X \in \Delta$.
  Let $f:X\to X$ be a minimal optimal map representing $\phi$. Let $\tau$ be a non-degenerate, candidate
  regular turn with respect to $\Delta$, which is also $f$-legal and let $\Delta_{\tau}$ be the
  simplex obtained by folding $\tau$\footnote{
    Note also that since $\tau$ is regular, it is in particular infinite non
    free, so $\Delta_\tau$ is different from $\Delta$. 
}. Then for
  all $Y\in\Delta_\tau$ we have $$\lambda_\phi(Y) >  \lambda(\phi).$$
\end{prop}

\begin{proof} From Corollary~\ref{FoldingCandidateRegular1} we know
  $\lambda_\phi(\Delta_\tau)\geq\lambda_\phi(\Delta)$, and if
  $\lambda_\phi(\Delta)>\lambda(\phi)$ the claim follows. Thus we may assume $\lambda_\phi(\Delta)=\lambda(\phi)$.
  
  For any $Z\in\Delta$ denote by $Z^t$ the point of $\Delta_\tau$ obtained from $Z$ by folding
  $\tau$ by an amount of $t$ (which is well defined, for any $Z$, for small enough $t$).
	
We will prove that there is an open neighbourhood $U$ of $X$ in $\Delta$ and 
$T>0$ (which depends only on $\phi$ and $X$) so that for all $Z \in U$ and $t<T$, the
point $Z^t$, is not minimally displaced. Then the result follows, as $U^T = \{ Z^t : Z\in U,
t<T \} $ is  an open neighbourhood of $X$ in $\Delta_{\tau}$, and by
Lemma~\ref{Cylinder}, if $\Delta_{\tau}$ were to contain a minimally displaced point, we would
be able to find a minimally displaced point in $U^T$, leading to a contradiction. Whence $\lambda_{\phi} (Y) >\lambda(\phi)$ for all $Y \in \Delta_{\tau}$.

We prove now our claim.  Since $\tau$ is candidate regular, in particular it is non-free. 
  By  Lemma~\ref{L4} there is a neighbourhood $U$ of $X$ in $\Delta$ so that for any
  point $Z \in U$, there is a minimal optimal map $f_Z: Z \to Z$, such that $\tau$ is
  $f_Z$-legal.  Clearly, for any such $U$ there is $T>0$ so that $Z^t$ is well defined for all $Z\in U$ and $t<T$.
By Corollary~\ref{C2} $\lambda_{\phi} (Z^t) \geq \lambda_{\phi} (Z)$, and if
$Z$ is not minimally displaced, the result follows.

  So, suppose that $Z \in \Min(\phi)$. In this case, since $f_Z$ is an optimal map
  representing $\phi$, and since $\phi$ is irreducible, then the tension graph of $f_Z$ is the
  whole $Z$ (Lemma\ref{wasfootnote}),
  and since $\lambda(\phi)>1$, Corollary~\ref{C2} applies in its strict inequality version.
In any case, $Z^t$ cannot be minimally displaced and our claim follows.
\end{proof}

\begin{rem}
  If one is interested in a version of Proposition~\ref{FoldingCandidateRegular2} for reducible
  automorphisms, one has just to add the hypothesis that $\tau$ is {\em stably} in the tension
  graph, that is to say, that $\tau$ is in the tension graph of any $f_Z$ for $Z$ close enough
  to $X$. This will be enough to apply Corollary~\ref{C2} as we did in the proof for
  irreducible automorphisms. 
\end{rem}

\begin{lem}\label{criticallemma}
  Let $[\phi]\in\Out(\G)$ be an irreducible element with $\lambda(\phi)>1$.
  Suppose that
  $X_1,X_2\in\Delta$ are two points minimally displaced by $\phi$ and let $f_1,f_2$ be train track
  representative of $\phi$ on $X_1$ and $X_2$ respectively.
  
  Suppose that $\tau$ is candidate regular turn. Then $\tau$ is $f_1$-legal in $X_1$ if and only
  if it is  $f_2$-legal in $X_2$.
  \end{lem}
  \begin{proof}
Suppose for contradiction that $\tau$ is $f_1$-legal but
$f_2$-illegal (in particular it is non-degenerate). Let $\Delta_\tau$ be the simplex obtained by folding $\tau$.
Since $\tau$ is candidate regular, we apply
Proposition~\ref{FoldingCandidateRegular2} (we can because train tracks are minimal optimal map
by Lemma~\ref{wasfootnote}), and we
have that $\lambda_{\phi}(Y) > \lambda(\phi)$ for any $Y \in \Delta_{\tau}$. On the other hand,
$\tau$ is $f_2$-illegal, and  $\Min(\phi)$ is invariant under isometrically folding
illegal turns (see \cite[~Theorem 8.23]{FM13}), which means that there is a point $Y \in
\Delta_{\tau}$ which is minimally displaced and that leads us to a contradiction. Clearly we
can switch the roles of $X_1$ and $X_2$, and the proof is complete.
\end{proof}

\begin{defn}[Simplex Critical Turns]\label{defnsc}
	Let $[\phi]\in\Out(\G)$ . For a simplex, $\Delta$, we say that a turn
        $\tau$ is {\bf simplex critical} if it is either candidate critical (see \ref{defncc}) or $f$-illegal for
        some train track  representative of $\phi$ defined on some point of $\Delta$ (if any).
	 A turn is called {\bf simplex regular} if it is not critical. The set of simplex
         critical turns is denoted by $\mathcal C_\Delta$.  Sometimes we will use the short
         notation  {\em $\Delta$-critical} to means {\em simplex critical in $\Delta$}.
\end{defn}

\begin{thm}
	\label{critical} If $[\phi]\in\Out(\G)$ is irreducible with $\lambda (\phi) > 1$, then
	the set $\mathcal{C}_{\Delta}$ is finite. In fact, it is sufficient to add to
        $\mathcal{C}_C(\Delta)$ the illegal turns for a single train track map (if any) to
        obtain the       whole of $\mathcal{C}_{\Delta}$.  
\end{thm}
\begin{proof}
The set $\mathcal C_C(\Delta)$ is clearly finite by construction (see Definition~\ref{defncc}), and if there is no minimal
displaced point in $\Delta$, we have nothing to prove. Otherwise, chose any train track
representative of $\phi$, $f$, defined on a point of $\Delta$. $f$-illegal turns are finitely
many (Lemma~\ref{L0}) and 
Lemma~\ref{criticallemma}
tells us that
$\mathcal{C}_{\Delta} = \mathcal{C}_C(\Delta) \cup \{\tau : \tau $ is an $f$-illegal turn$\}$ and, in particular, it is finite.
\end{proof}

\begin{rem}

It is worth mentioning that simplex regular turns can be effectively detected, having a train
track map $f$ in hand. Namely, suppose
that a turn $\tau$ is
\begin{enumerate}
\item not a free turn nor a turn with finite vertex group; and
\item not the $f$-image of an edge; and
\item not the $f$-image of a free turn; and
\item not the $f$-image of a turn involving group elements in $H$; and
\item not $f$-illegal;
\end{enumerate}
then $\tau$ is simplex regular. In particular, Proposition~\ref{FoldingCandidateRegular2} tells
us that if we have $X\in\Min(\phi)$ and we want to
find all neighbours of $X$ obtained from $X$ by a single turn-fold, and  which still are in
$\Min(\phi)$, then we only need to check turns in the finite complement of the above effective list,
namely turns that are either
\begin{enumerate}
\item free or with finite vertex group; or
\item in the $f$-image of an edge; or
\item the $f$-image of a free turn; or
\item the $f$-image of a turn involving group elements in $H$; or
\item $f$-illegal.
\end{enumerate}

\end{rem}

\section{Folding and unfolding collapsed forests}
Here we extend the unfolding construction of Section~\ref{s3} to the general case of two
simplices, one face of the other. We remind that we always understand that a
straight map between elements of $\O(\G)$ is an $\O$-map (i.e. $G$-equivariant at level of trees).

A straight map $p:X\to Y$, defines on $X$ a simplicial structure $\sigma_p$, by pulling back
that of $Y$. With respect to $\sigma_p$, the map $p$ is tautologically simplicial. 
We define the {\bf simplicial volume} of $p$, $\operatorname{svol}(p)$ as the number of edges of $\sigma_p$. 

If in addiction $p$ is {\bf locally isometric on edges}, then it defines (some) folding paths $X=X_0,\dots,X_n=Y$
obtained by recursively identifying pairs of edges of $\sigma_p$ having a common vertex and the same
$p$-image. Together with the $X_i$ there are quotient maps $q_i:X_{i-1}\to X_i$ given by the
identification, and maps $p_i:X\to Y$ defined by $p_i(x)=p(q_i^{-1}(x))$. (Note that $p_0=p$ and $p_n=id$).
We refer to any folding path obtained as above as a {\bf folding path directed by $p$}. We
say that $X=X_0,\dots,X_n=Y$ has length $n$. 
\begin{lem}  
	\label{redsimpvol}
	Let $X, Y \in \O(\G)$ and $p:X \to Y $ be a straight map which is locally isometric on
        edges. Then any folding path directed by $p$ has length at most $\operatorname{svol}(p)$.
\end{lem}
\begin{proof}
  At any step the number of edges decreases by one.
\end{proof}

\begin{lem}
	\label{unfold}
	Let $\Delta,\Delta'$ be simplices of $\O(\G)$ so that 
	$\Delta$ is a face of $\Delta'$. For any $Y \in \Delta'$ there is a point,
	$\unf(Y) \in \Delta$ and a  straight map, $p: \unf(Y) \to Y$,  such that:
	\begin{enumerate}
		\item $p$ is a local  isometry on edges;
		\item If $v$ is a non-free vertex in $Y$, then $p^{-1}(v)$ is a single vertex;
		\item $\operatorname{svol}(p)$ is at most $2D(\Delta')^2$, where $D(\Delta')$ is
		the number of edges of $\Delta'$.  
	\end{enumerate}
	Moreover, any folding path directed by $p$ produces maps which still satisfy $(1)$
	and $(2)$.
\end{lem}

\begin{proof}
	The underlying graph $X$ of $\Delta$ is obtained by the collapse in $Y$ of a simplicial forest
	$\F=T_0\sqcup\dots\sqcup T_k$ each of whose tree $T_i$ contains at most one non-free
	vertex. We define $\unf(Y)$ by isometrically unfolding each tree.  More precisely, for any
	$T_i$ we choose a root-vertex $w_i$ with the requirement that $w_i$ is the unique
	non-free vertex of $T_i$, if any. For any leaf $y$ of the forest, say $y$ is a leaf of $T_i$,
	there is a unique path $\gamma_y$ connecting $y$ to $w_i$ in $T_i$. For notational convenience
	we define $\gamma_y$ to be the constant path for any other vertex of $Y$.
	
	The metric on $X$ defining the point $\unf(Y)$ is given as follows. Any edge $e$ of $X$ has a
	preimage in $Y$ which is also an edge. We declare $$L_{\unf(Y)}(e)=L_{Y}(\gamma_a)+L_Y(e)+L_Y(\gamma_b)$$ 
	where $a,b$ are the endpoints of the preimage of $e$ in $Y$.
	As an oriented edge, $e$ is therefore the concatenation of three sub-segments $$e=A\cdot
	E\cdot B$$ of lengths $L_{Y}(\gamma_a), L_Y(e),$ and  $L_Y(\gamma_b)$ respectively. The map
	$p$ is now defined by isometrically identifying $A$ with $\gamma_a$, $E$ with the copy of $e$
	in $Y$, and $B$ with $\gamma_b$. The union of all $A$-segments and $B$-segments form a forest
	which can be viewed as an isometric unfolding of $\F$, and out of that forest, $p$ is  basically
	the identity by definition. Conditions $(1)$ and $(2)$ immediately follow. As for $(3)$, it
	suffices to note that for any $y\in Y$, the cardinality of $p^{-1}(y)$ is bounded by the number
	of leaves of $\F$, which is bounded by $2D(\Delta')$. Therefore $\sigma_p$ has at most
	$2D(\Delta)D(\Delta')\leq 2D(\Delta')^2$ edges.  The last claim is easy to verify and we leave it to the reader. 
\end{proof}

\begin{lem}
	\label{cancel} Let $X, Y \in \O(\G)$. 
	Let $p:X \to Y$ be a straight map such that if $v$ is a non-free vertex in $Y$, then $p^{-1}(v)$ is a single vertex in $X$ (condition $(2)$ in Lemma~\ref{unfold}). 
	
	Let $\alpha, \beta$ be paths in $X$, starting at the same vertex, and so that $\bar\alpha\beta$
	is reduced. If $p(\alpha) = p(\beta)$, then the only non-free vertex
	crossed by each of them, if any, is their initial vertex (which is crossed only once).  
\end{lem}
\begin{proof}
	We argue by contradiction assuming that $\alpha$ crosses a non-free vertex other than its
	initial point (we consider multiple crossings of the same vertex as distinct crossings). Up to possibly truncating $\alpha$ and $\beta$, we may assume that the last
	vertex $w$ of $\alpha$ is non-free, and that $\alpha$ crosses no other non-free vertex
	except possibly its initial point. Then our assumption on $p$ implies that the last vertex
	of $\beta$ must be $w$. But in this case  $\alpha \bar \beta$ would  define a non trivial group-element which is collapsed by $p$, contradicting that $X,Y$ are in the same deformation space.
\end{proof}

 \begin{prop}	\label{foldable}
  Let $[\phi]\in\Out(\G)$ be an irreducible element with $\lambda(\phi)>1$.
    Let $X, Y \in \O(\G)$, $X \neq Y$, and suppose that there is $p:X \to Y$ a straight map such that 
    	\begin{enumerate}
    	\item $p$ is a local isometry on edges;
    	\item if $v$ is a non-free vertex in $Y$, then $p^{-1}(v)$ is a single vertex.
    \end{enumerate}

  	   If any $p$-illegal turn is simplex regular (Definition~\ref{defnsc}),  then $\lambda_\phi(Y)>\lambda(\phi)$.  
 \end{prop}
 \begin{proof}
     By Lemma~\ref{L0}, the set of $p$-illegal turn is a finite set $\{\tau_1,\dots,\tau_k\}$. Since $X \neq Y$, this set is non-empty.

Denote $\Delta$ the simplex of $X$.     Let $f:X\to X$ be a minimal
  optimal map representing $\phi$. Note that for any element of $G$, seen as a loop
  $\gamma$ in $X$, the loop $p(f(\gamma))$ represents $\phi(p(\gamma))$ in $Y$.

  Firstly, we deal with the case where $X \notin \Min(\phi)$. By Remark~\ref{RB3} and Lemma~\ref{LB1} there is an $f$-legal loop $\gamma \in A_{\Delta}$ in the tension
graph of $f$. Its image, $f(\gamma)$, crosses only candidate critical turns by definition of
candidate critical and, in particular, it doesn't cross any of
the $\tau_i$'s. In other words, $f(\gamma)$ is $p$-legal and therefore, as $p$ is a local 
isometry on every edge, $L_Y(p(f(\gamma))) = L_X(f(\gamma))$.
On the other hand, again because $p$ is a local isometry, $L_Y(p(\gamma)) \leq L_X(\gamma)$
which means that $$\lambda_{\phi}(Y) \geq \frac{L_Y(p(f(\gamma)))}{L_Y(p(\gamma))} \geq
\frac{L_X(f(\gamma))}{L_X(\gamma)} = \lambda_{\phi} (X) > \lambda(\phi).$$

Suppose now $X\in\Min(\phi)$. In this case we may assume that $f$ is a train track map
representing $\phi$ (in particular $f$ is a minimal optimal map, see Section~\ref{definitions}). All $\tau_i$ are $f$-legal because of simplex-regularity.

Let's first assume that there is some $\tau_j$ which is mapped by $f$ to a turn distinct from any of the $\tau_i$'s. Then, by Corollary~\ref{CandLegal2}, there is an $f$-legal loop $\gamma$ (which is in the tension graph) such that, 

\begin{enumerate}
	\item $\sum_{i=1}^k \#(\gamma,\tau_i) = 1$,
	\item $\sum_{i=1}^k \#(f(\gamma),\tau_i) =0$.
\end{enumerate}

Hence, $L_Y(p(\gamma)) < L_X(\gamma)$ whereas, $L_Y(p(f(\gamma))) = L_X(f(\gamma))$. Hence, 

$$
\lambda_{\phi}(Y) \geq \frac{L_Y(p(f(\gamma)))}{L_Y(p(\gamma))} > \frac{L_X(f(\gamma))}{L_X(\gamma)} = \frac{\lambda_\phi(X)L_X(\gamma)}{L_X(\gamma)}   = \lambda (\phi). 
$$

Otherwise, $f$ must leave invariant the set of $\tau_i$. We will now work with the
$\langle\sim_{f^k}\rangle$-legal structure, in order to ensure that the image of a legal loop
is again legal. Since all $\tau_i$ are legal and set of $\tau_i$ is invariant under the action of $f$, then they also are
$\langle\sim_{f^k}\rangle$-legal.

Let $\Sigma$ be the set of $\langle\sim_{f^k}\rangle$-legal loops $\gamma$ in $X$ that satisfy
$$\sum_{i=1}^k\#(\gamma,\tau_i)=1.$$

By Corollary~\ref{candlegal3}, the set $\Sigma$ is not empty.
Let $$C=\sup\{L_X(\gamma)-L_Y(p(\gamma)): \gamma\in \Sigma\} .$$
The Bounded Cancellation Lemma (see for
instance~\cite[Proposition~3.12]{Horbez}) and discreteness show that $C$ is a maximum, which
means that $C$ is realised by some loop $\gamma_C$. Moreover, since the $\tau_i$'s are
$p$-illegal, $C>0$. We claim that $\gamma_C$ can be chosen so that $f(\gamma_C)$ also belongs
to $\Sigma$. This will be enough as, since $\gamma_C$ realises $C$, then $L_Y(p(\gamma_C))= L_X(\gamma_C)-C$, while
$L_Y(p(f(\gamma_C)))\geq L_X(f(\gamma_C))-C$ (because $f(\gamma_C) \in \Sigma$). But then

 $$
\lambda_{\phi}(Y) \geq \frac{L_Y(p(f(\gamma_C)))}{L_Y(p(\gamma_C))} \geq \frac{L_X(f(\gamma_C)) -
	C}{L_X(\gamma_C) - C} = \frac{\lambda_\phi(X)L_X(\gamma_C) - C}{L_X(\gamma_C)
	- C}   > \lambda (\phi) 
$$

where the strict inequality follows from the fact that $\lambda_\phi(X)= \lambda(\phi)  > 1$.

We prove now our claim. Consider any $\gamma\in\Sigma$ realising the maximum $C$.
By Proposition~\ref{CandLegal0} $\gamma$ can be modified via turn surgeries to a
$\langle\sim_{f^k}\rangle$-legal  turn $\gamma'$ 
such that $\sum_i\#(f(\gamma'),\tau_i)=1$. Note that such surgeries occur only at
non-free vertices.

It remains to show that the performed surgeries do not affect the $p$-cancellation of the
original loop $\gamma$. As $\tau_j$ is the unique $p$-illegal of $\gamma$, there exist
sub-paths $\alpha, \beta$ of $\gamma$ so that $p(\alpha) = p(\beta)$, the first edge of
$\alpha$ together with the first edge of $\beta$ form the turn $\tau_j$, and $L_X(\alpha) =
L_X(\beta) = C/2$. That is, $\alpha$ and $\beta$ are the sub-paths of $\gamma$ which realise
the $p$-cancellation. By Lemma~\ref{cancel}, both $\alpha$ and $\beta$ cross only free turns
and so the performed surgeries did not affect neither $\alpha$, nor $\beta$. As the turn
$\tau_j$ is not affected by the surgeries, as well, it follows that the $p$-cancellation of
$\gamma'$, is the same as the $p$-cancellation of $\gamma$, that is to say
$$ L_X(\gamma')-L_Y(p(\gamma')) = L_X(\gamma)- L_Y(p(\gamma))=C$$
as we wanted.

\end{proof}

\section{Exploring the Minset}

Proposition~\ref{FoldingCandidateRegular2} tells us that if we want to travel along
$\Min(\phi)$,  we have to perform only simplex critical turns. Given two simplices
$\Delta,\Delta^1$, one face on the other, we can easily go from
$\Delta$ to $\Delta^1$ in few steps by simple folds. However, even if both simplices intersect
$\Min(\phi)$, such folds need not necessarily to be  
simplex critical. Nonetheless, it may exists a, a priori longer, folding path between them that uses only
simplex critical folds.

\begin{defn}
  Let $[\phi]\in\Out(\G)$ and $\Delta$ be a simplex in $\O(\G)$.
  We denote by the {\bf simplex critical neighbourhood} of $\Delta$ of radius 1, all the
  simplices of $\O(\G)$ which can be obtained from $\Delta$ via a simplex critical fold,
  including $\Delta$ itself. 
	
	We denote by the simplex critical neighbourhood of $\Delta$ of radius $n+1$, the union of all the simplex critical neighbourhoods of radius 1, of all simplices in the  simplex critical neighbourhood of $\Delta$ of radius $n$. 
\end{defn}

\begin{rem}
	By Theorem~\ref{critical}, if $\phi$ is irreducible with $\lambda(\phi) > 1$, then any simplex critical neighbourhood of finite radius consists of finitely many simplices.
\end{rem}

\begin{defn}
	Let $X$ be a point of $\O(\G)$ and let's denote by $\Delta = \Delta_X$ the
        corresponding simplex. The dimension of $D(\Delta)$ of $\Delta$ is the number of edges
        of $X$.  
	We denote by $D=D(\G)$ the dimension of $\O(\G)$, i.e. the maximum
	number of edges we see in elements of $\O(\G)$.
	
	In $\O_1(\G)$ the dimension of the simplex containing $X$ is one less, as is the dimension of the entire space.
\end{defn} 
\begin{thm}
	\label{locfin}
	Let $[\phi]\in\Out(\G)$ be an irreducible automorphism with $\lambda (\phi)  > 1$.  
	Let $\Delta, \Delta^1$ be open simplices of $\O(\G)$, both intersecting
        $\Min(\phi)$, and such that $\Delta$ is a face of $\Delta^1$. Then, $\Delta^1$ is
        contained in the simplex  critical neighbourhood of $\Delta$ of radius $2 D(\G)^2$.  
	In particular, $\Min(\phi)$ is locally finite.
	
	We immediately deduce the same statement for $\O_1(\G)$.

\end{thm}
\begin{proof}
  We will show there exists a sequence of simplices, 
	$\Delta_0, \Delta_1, \ldots, \Delta_n$, with $\Delta_0 = \Delta$ and $\Delta_n =
        \Delta^1$, where $n \leq 2 (D(\G))^2 $, and such that each $\Delta_{i+1}$ is obtained by
        folding a $\Delta_i$-simplex critical turn.
        	
	The underlying graph of $\Delta$ is obtained from that of $\Delta^1$ by collapsing a
        forest $\F$ each of whose tree contains at most one non-free vertex. We apply
        Lemma~\ref{unfold}, to get a straight map, $p:\unf(Y) \to Y$ which is locally isometric
        on edges. Subdivide $\unf(Y)$ so that the $p$-image of each subdivided edge is a single edge in $Y$. 
        Proposition~\ref{foldable} tells us that it must exist a $p$-illegal turn that is also simplex critical. Fold this turn; this is an isometric fold directed by $p$, since $p$ is an isometry on edges, and we get a map  $p_1:X_1\to
        Y$ which, by Lemma~\ref{unfold}, satisfies conditions $(1)$ and $(2)$ of
        Proposition~\ref{foldable}, which therefore applies. This process recursively defines a
        folding path from $\unf(Y)$ to $Y$, directed by $p$. The length of such folding path is
        bounded by $2D(\Delta^1)^2$ by Lemmas~\ref{redsimpvol} and~\ref{unfold}.

	To conclude the proof, note that any simplex of $\O(\G)$ adjacent to $\Delta$ and
        lying in $\Min(\phi)$ either has $\Delta$ as a face, or is a face of $\Delta$. The
        first case is dealt with above, and in the second case, we know that there are at most
        $2^{D(\Delta)} \leq 2 ^{D(\G)}$ faces.

	The statement for $\O_1(\G)$ follows since the displacement function is invariant under change of volume.
	
\end{proof}

We now show that Theorem~\ref{locfin} may be strengthened to show that the minimally displaced
set is {\em uniformly} locally finite. That is, there is a uniform bound (depending only on
$\lambda(\phi)$ and $D(\G)$) on the number of
simplices, adjacent to a given simplex in $\Min(\phi)$ which are also in $\Min(\phi)$. In what
follows we are not focused in optimal bounds.

\begin{defn}
	Let $\Delta$ be a simplex in $\O(\G)$. Then we define the centre $X_{\Delta} \in
        \Delta$ to be the graph where all edges have the same length. 	
\end{defn}
Since we are interested in the function
        $\lambda_{\phi}$, which is scale invariant, we may scale the metric on 
        $X_{\Delta}$ as we wish; we will therefore decree that all the edges of $X_{\Delta}$
        have length $1$.  

\begin{lem}\label{Lin} Let $[\phi]\in\Out(\G)$. For any $X,Y\in \O(\G)$ we have
$$\frac{\lambda_\phi(X)}{\lambda_\phi(Y)}\leq\Lambda(X,Y)\Lambda(Y,X).$$
\end{lem}
\begin{proof}
This immediately follows from the non-symmetric {\em triangle inequality}: 
	$$
	\lambda_{\phi}(X)=\Lambda(X,\phi X) \leq    \Lambda(X,Y) \Lambda(Y,\phi Y)
        \Lambda(\phi Y, \phi X)=\Lambda(X, Y)  \lambda_{\phi}(Y)  \Lambda( Y,  X).
       $$

\end{proof}

As above,
 $D=D(\G)=\dim(\O(\G))$ is the maximum number of (orbits of) edges we see in elements of $\O(\G)$.

\begin{lem}
	\label{numfold}
	Let $[\phi] \in Out(\G)$, and $\Delta, \Delta'$ be simplices in $\O(\G)$ such
        that $\Delta$ is a face of $\Delta'$. Then
	$$
	\frac{1}{2D}\lambda_{\phi}(X_{\Delta}) \leq \lambda_{\phi}(X_{\Delta'})  \leq 2D\lambda_{\phi}(X_{\Delta}).
	$$
\end{lem}

\begin{proof}
By Lemma~\ref{Lin}, it is sufficient to prove that
	$$
	\Lambda(X_{\Delta}, X_{\Delta'}) \Lambda(X_{\Delta'}, X_{\Delta}) \leq 2D.
	$$
	
	Since the product on the left is scale invariant (even though each factor is not) we
        are free to choose the volumes for each of the points. Specifically, the two centres
        have different volumes, as we set every edge to have length $1$.
        In particular, 
	$$
	 \Lambda(X_{\Delta'}, X_{\Delta}) \leq 1, 
	$$
	as loops become shorter when we collapse a forest (since the length in each case is a count of the number of edges). 
	
	To complete the argument, we consider the map $p:\unf(X_{\Delta'})\to
        X_{\Delta'}$ given by Lemma~\ref{unfold}, and we read it as a map from $X_\Delta\to X_{\Delta'}$.
        It is easy to see that the image of an edge under this map cannot cross the same edge more than twice (usually no more than once, but twice may happen if
        $\Delta$ has some edge-loop). It follows that $L_{X_{\Delta'}}(p(\gamma))\leq
        2DL_X(\gamma)$. Hence $\Lambda(X_\Delta,X_{\Delta'})\leq 2D$.
\end{proof}

 The maximum number of vertices we see in elements of $\O(\G)$ is bounded by $2D$. 
Denote by $M=M(\G)$ the maximal cardinality of finite vertex groups. Set $K=K(\G)=D+M+1$.
\begin{lem}
	\label{numcrit}
	Let $[\phi]\in\Out(\G)$, $\Delta$ be a simplex of $\O(\G)$,  and
        $\mathcal{C}_{\Delta}$ be   the corresponding set of simplex critical
        turns. Then 
	$$
	|\mathcal{C}_{\Delta}| \leq(10K)!\lambda_{\phi}(X_{\Delta})
	$$
\end{lem}
\begin{proof}
	Since all the edges of $X_{\Delta}$ have length $1$, the number of turns crossed by a
        loop is then equal to 
        its length in  $X_{\Delta}$. Hence, the number of turns crossed by an element of
        $A_\Delta$ is bounded above by 
	$$
	\lambda_{\phi}(X_{\Delta}) 4 D  |A_{\Delta}|,
	$$
	where the term $4 D$ appears as the maximum length of a loop in $A_{\Delta}$, as read in $X_{\Delta}$ (see Definition~\ref{ADelta}). Note that the length in $X_{\Delta}$ is simply the number of edges.

	To estimate $|A_{\Delta}|$, we count the number of sequences of edges, where each edge
        appears at most $4$ times - ignoring the incidence relations to simplify matters. The
        number of sequences of $n$ objects of length $k$, is $n!/(n-k)! \leq n!$, and so the
        number of sequences of $n$ objects of length at most $n$ is bounded by $(n+1)!$. For
        building a loop in $A_\Delta$ we have $D$ edges, each of which appear at
        most $4$ times and, taking in account the group elements, the total number of objects
        we can use to write an element of $A_\Delta$ is $8D$. Hence
	$$
	|A_{\Delta}| \leq (8D+1)!
	$$
	
	This only bounds the number of turns crossed by elements of $A_\Delta$. 
        In order to bound the simplex critical turns, we need to add all turns based at
        vertices with finite vertex group, and all
        the illegal turns for a putative train track map. The former is bounded by the number
        of possible pairs of germs of  edges multiplied by the cardinality of finite vertex
        group, hence 
        by $M(\G)D^2$. The latter, because of Lemma~\ref{L0}, is bounded 
        by the number of pairs of edges; hence by $D^2$. In
        total we have 
        $$\lambda_\phi(X_\Delta)4D(8D+1)!+(M+1)D^2\leq \lambda_\phi(X_\Delta)4K(8K)!+K^3$$

        and the result follows.  
\end{proof}

\begin{lem}
	\label{numirred}
	Let $[\phi]\in\O(\G)$ be an irreducible element, and suppose that $\Delta$ is an
        open simplex of $\O(\G)$ which contains a point of $\Min(\phi)$. Then, 
	$$
	\lambda_{\phi}(X_{\Delta}) \leq 9 \dim\O(\G)\lambda(\phi)^{3D+2}.
	$$
\end{lem}
\begin{proof}
	Let $X_{\min}$ denote a point of $\Delta$ which is minimally displaced by $\phi$. 
	We will use the fact that $X_{\min}$ is $\epsilon$-thick, as in 
        \cite[Proposition~10]{BestvinaBers} (it is proved there in $CV_n$, but the proof is the same in this
        context, see also~\cite[Section 8]{FM13}). That is, since $\phi$ is irreducible, there
        is a lower bound on $L_X(\gamma)/\vol(X)$ that
        depends only on  $\lambda(\phi)$ (and not on $X\in\O(\G)$ nor on the non-elliptic
        element 
        $\gamma$). Concretely, this lower bound can be taken to be the reciprocal of
        $C(\phi)=3\dim(\O(\G))\lambda(\phi)^{3\dim(\O(\G))+1}$. If we normalise
        $X_{\min}$ to have volume $1$ then $\Lambda(X_\Delta,X_{\min})\leq 1$. Moreover,
        since the stretching factor
        $\Lambda$ is realised by candidate loops of simplicial length at most $3$ (by the Sausage
        Lemma~\cite[Theorem~$9.10$]{FM13}), 
	we can then deduce that, 
	$$
	\Lambda(X_{\min}, X_{\Delta}) \Lambda(X_{\Delta}, X_{\min}) \leq
        3C(\phi) \Lambda(X_{\Delta}, X_{\min})\leq 3C(\phi).
	$$
	
	The result now follows from Lemma~\ref{Lin}.

\end{proof}

\begin{cor}
	\label{unif}
		Let $[\phi]\in\Out(\G)$ be an irreducible element with $\lambda (\phi)  > 1$. Then $\Min(\phi)$ is uniformly  locally finite.
\end{cor}
\begin{proof}
	By Theorem~\ref{locfin}, it is sufficient to show that the simplex critical
        neighbourhood $N$ of radius $2D^2$, of a minimally displaced simplex $\Delta_0$,
        contains a uniformly bounded number of simplices (the number of faces of a simplex is
        always uniformly bounded). 
	Hence it suffices to uniformly bound the cardinality of $\mathcal C_\Delta$  for each
        simplex we encounter in $N$.

        By Lemma~\ref{numfold}, $\lambda_\phi(X_\Delta)\leq
        (2D)^{2D^2}\lambda_{\phi}(X_{\Delta_0})$, which is uniformly bounded by Lemma~\ref{numirred}.
        Lemma~\ref{numcrit} completes the proof.
	\end{proof}

\section{Definitions and basic results used in the paper}
\label{definitions}
Our notation and definitions are quite standard. We briefly recall them here, referring the reader
to~\cite{FM18I} for a detailed discussion.

\begin{defn}
  A free splitting $\G$ of a group $G$ is a decomposition of $G$ as a free product $G_1*\cdots
  *G_k*F_n$ where $F_n$ is the free group of rank $n$. We admit the {\em trivial} splitting
  $G=F_n$. We do not require that the groups $G_i$'s are indecomposable. 
\end{defn}

\begin{defn}
  A simplicial $G$-tree is a simplicial tree $T$ endowed with a faithful simplicial action of
  $G$. $T$ is minimal if it  has no proper $G$-invariant sub-tree. A $G$-graph is a graph of
  groups whose fundamental group, as graph of groups, is isomorphic to $G$. The action of $G$
  on a $G$-tree, is called marking.
  \end{defn}

  \begin{defn}\label{G-graphs}
    Let $\G$ be a free splitting of $G$. In terms of Bass-Serre theory, a $\G$-tree is
    the tree dual to $\G$, and  a $\G$-graph is the corresponding graphs of groups.
    More explicitly, a 
      $\G$-tree is a simplicial $G$-tree $T$ such that:
   \begin{itemize}
  \item For every $G_i$ there is exactly one orbit of vertices whose stabilizer is conjugate to
    $G_i$. Such vertices are called {\em non-free}. Other vertices are called {\em free}.
  \item $T$ has trivial edge stabilisers.
  \end{itemize}
A $\G$-graph dual to a $\G$-tree. Namely, it is a finite connected $G$-graph of groups $X$, along with an isomorphism, $\Psi_X: G \to \pi_1(X)$ - a marking - such that:
\begin{itemize}
\item $X$ has trivial edge-groups;
\item the fundamental group of $X$ as a topological space is $F_n$;
\item the splitting given by the vertex groups is equivalent, via $\Psi_X$, to $\G$. That is, $\Psi_X$ restricts to a bijection from the conjugacy classes of the $G_i$ to the vertex groups of $X$.
\end{itemize}
\end{defn}

The universal cover of a $\G$-graph is a $\G$-tree and the $G$-quotient of a $\G$-tree is a $\G$-graph.

\begin{defn}
Let $\G$ be a splitting of a group $G$. The Outer Space of $\G$, also known as deformation space of
$\G$, and denoted $\O(\G)$ is the set of classes of minimal, simplicial, metric $\G$-graphs $X$ with no
redundant vertex. (The equivalence relation is
given by $G$-isometries.)    

We denote by $\O_1(\G)$ the volume $1$ subset of $\O(\G)$.

\end{defn}
$\O(\G)$ can be regarded also as set of $\G$-trees, but in the present paper we adopted the
graph view-point.
Given a graph of groups $X$ with trivial edge groups, we denote by $\O(X)$
the corresponding deformation space $\O(\pi_1(X))$ (we notice that $X\in\O(X)$ if $X$ is a core
graph with no redundant vertex) where $\pi_1(X)$ is endowed with the splitting given by vertex groups.
We refer the reader to~\cite{FM13,FM18I,GuirardelLevitt} for more details
on deformation spaces.

\begin{defn}
  Let $X\in\O(\G)$. The simplex $\Delta_X$ is the set of marked metric graphs obtained from
  $X$ by just changing edge-lengths. Since edge-lengths are strictly positive we think $\Delta_X$
  as (a cone over) an open simplex. Given a simplex $\Delta$ one can consider the 
  the closure of $\Delta\in \O(\G)$ or its simplicial bordification. Namely, faces of
  $\Delta$ come in two flavours: that in $\O(\G)$, called finitary faces,
  and that not in $\O(\G)$ (typically in other deformation spaces) called faces at infinity. 
  
  We also have a simplex $\Delta_1(X)$ in $\O_1(\G)$ - the intersection of $\Delta(X)$ with
  $\O_1(\G)$ - which is a standard  open simplex of one dimension less.
\end{defn}

There are two topologies on $\O(\G)$, both of which restrict to the Euclidean topology on each simplex; these are the weak topology and the axes or Gromov-Hausdorff topology. The topology  induced by the Lipschitz metric is the latter one.

\begin{defn}
    Let $G$ be endowed with the splitting $\G:G=G_1*\dots*G_i*F_n$.
The group of automorphisms of $G$
  that preserve the set of conjugacy classes of the $G_i$'s is
  denoted by $\operatorname{Aut}(\G)$. We set
  $\operatorname{Out}(\G)=\operatorname{Aut}(\G)/\operatorname{Inn}(G)$.
\end{defn}

The group $\Aut(\G)$ naturally acts on $\O(\G)$ by precomposition on marking, and
$\operatorname{Inn}(\G)$ acts trivially, so $\Out(\G)$ acts on $\O(\G)$.

Since the volume is invariant under this action, we also get an action of $\Aut(\G)$ and  $\Out(\G)$ on $\O_1(\G)$.

\begin{defn} Given a splitting $\G$ of $G$, and  $X,Y\in\O(\G)$, a map $f:X\to  Y$ is called an
  $\O$-map at the level of the tree, if it is Lipschitz-continuous and $G$-equivariant (resp., a $\G$-map between graph of groups, it called an $\O$-map if it is the projection of an $\O$-map at level of the universal covers). The Lipschitz constant of $f$ is denoted by $\Lip(f)$.

\end{defn}

\begin{defn}
Let $X,Y$ be two metric graphs.
A Lipschitz-continuous  map $f:X\to Y$ is 
{\em straight} if it has constant speed on edges, that is to say, for any edge $e$
of $X$ there is a non-negative number $\lambda_e(f)$ such that edge $e$ is uniformly stretched
by a factor $\lambda_e(f)$.
 A straight map between elements of $\O(\G)$ is always supposed to be an $\O$-map.
\end{defn}

\begin{rem} $\O$-maps always exist and the images of non-free vertices are
  determined a priori by equivariance 
  (see for instance~\cite{FM13}).
  For any $\O$-map $f$ there is a unique straight map denoted by $\PL(f)$, which is homotopic,
  relative to vertices, to $f$. We have $\Lip(\PL(f))\leq \Lip(f)$.
\end{rem}

\begin{defn}\label{deftg}
  Let $f:X\to Y$ be a straight map. We set $\lambda_{\max}(f)=\max_{e}\lambda_e(f)=\Lip(f)$ and 
  define the {\em tension  graph} of $f$ as the set
$$\{e \text{ edge of } X : \lambda_e(f)=\lambda_{\max}\}.$$
\end{defn}

\begin{defn}
  Given $X,Y\in\O(\G)$ we define $\Lambda(X,Y)$ as the infimum of Lipschitz constants of
  $\O$-maps from $X$ to $Y$. That inf is in fact a minimum and coincides with $\max_\gamma
  \frac{L_Y(\gamma)}{L_X(\gamma))}$ where $\gamma$ runs on the set of loops in $X$ (seen as a
  graph). (See for instance~\cite{FM11,FM13,FM18I}.)
\end{defn}

\begin{defn}
A gate structure on a graph of groups $X$ is a $G$-equivariant equivalence relation of germs of edges at
vertices of the universal cover, $\widetilde{X}$, of $X$. A \textit{train track structure} on a graph of groups $X$ is a gate structure on $X$ with at least two gates at each vertex.
\end{defn}

\begin{rem}
For a straight map $f:X\to Y$, we consider two different gate structures, which we denote by
 $\sim_f$ and $\langle\sim_{f^k}\rangle$, the latter being defined only if $X=Y$. Two germs of $X$ are $\sim_f$-equivalent, if they have the same  non-collapsed $f$-image and they are $\langle\sim_{f^k}\rangle$-equivalent, if they have the same non-collapsed $f^k$-image for some positive integer $k$. 

This second gate structure -  $\langle\sim_{f^k}\rangle$ - only makes sense if $Y$ is the same topological object as $X$, so that we may iterate $f$. However, $Y$ will usually be a different point of $\O(\G)$ since the $G$-action will be different.

We refer to $\sim_f$, as the gate structure which is induced by $f$.
\end{rem} 

\begin{defn}
  A turn of $X\in\O(\G)$ is (the $G_v$-orbit of) an unoriented pair of germs of edges based
  at a vertex $v$ of $X$.  A turn is legal if its germs are not in the same gate. A simplicial
  path in $X$ is legal if it crosses only legal turns. Legality here depends on the gate structure, which for us will either be the $\sim_f$ or $\langle\sim_{f^k}\rangle$ structure for some $\O$-map $f$ with domain $X$.
\end{defn}

\begin{defn}\label{defnmo}
  Let $X,Y\in\O(\G)$. A straight map $f:X\to Y$ is said to be optimal if
  $\Lip(f)=\Lambda(X,Y)$ and every vertex of the tension graph is at least two-gated (i.e. the gate structure $\sim_f$ is a train track structure on the tension graph). An optimal map is minimal if every edge of the tension graph extends to a legal loop in the tension graph (not all optimal maps are minimal, but minimal optimal maps always exist (see~\cite{FM18I})).
\end{defn}

\begin{defn}
  Given $[\phi]\in\Out(\G)$ and $X\in\O(\G)$ we say that an $\O$-map $f:X\to \phi X$
  represents $\phi$. 
 Note  that $X$ and $\phi X$ are the same graph with different markings, so we sometimes abuse notation by saying that $f$ is a map $f: X \to X$ which represents $\phi$.
  In this situation we can speak of the $\langle\sim_{f^k}\rangle$ gate structure. 
\end{defn}
  Any $\phi$ is represented by a minimal optimal map (see~\cite{FM18I}).

\begin{defn}
	We call $[\phi] \in \Out(\G)$ {\em reducible} if there exists an $\O$ map $f: X \to X$
        representing $\phi$, a lift $\wt f:\wt X\to \wt X$ and a $G$-subforest, $Y \subsetneq \wt
        X$ which is $\wt f$-invariant and contains the axis of a hyperbolic element. Otherwise $[\phi]$ is called irreducible.
\end{defn}

\begin{rem}
	An automorphism $\phi$ is called {\em iwip} - irreducible with irreducible powers - if every positive iterate of $\phi$ is irreducible. We mention this for completeness, but we are concerned with the general irreducible class for this paper.
\end{rem}

\begin{defn}
A straight map $f:X\to X$ representing $\phi$ is a train track map, if there is a train track structure on $X$ so that:
\begin{itemize}
	\item  $f$ maps edges to legal paths
	\item  If $f(v)$ is a vertex, then $f$ maps inequivalent germs at $v$ to inequivalent
	germs at $f(v)$.
\end{itemize}
\end{defn}

\begin{rem} (See~\cite{FM13, FM18I} for more details):
\begin{enumerate}
\item If  $f:X\to X$ is train track map representing $\phi$ (with respect to some gate structure), then it is a train track map with respect to the $\langle\sim_{f^k}\rangle$ gate structure.
\item Irreducible elements of $\Out(\G)$ admit train track representatives. 
\end{enumerate}
\end{rem}

\begin{defn}
  Given $[\phi]\in\Out(\G)$ the displacement function $$\lambda_\phi:\O(\G)\to\O(\G)$$
  is defined by $$\lambda_\phi(X)=\Lambda(X,\phi X).$$ If $f:X\to X$ is any $\O$-map
  representing $\phi$, then $$\lambda_\phi(X)=\sup_{\gamma}\frac{L_X(\phi\gamma)}{L_X(\gamma)}$$
where the sup is taken over all loops $\gamma$ in $X$ (and it is actually a maximun
by~\cite{FM13,FM18I}) and $L_X(\gamma)$ denotes the reduced length of $\gamma$. 
For a simplex $\Delta$ we define $$\lambda_\phi(\Delta)=\inf_{X\in\Delta}\lambda_\phi(X)$$
and $$\lambda(\phi)=\inf_{X\in\O(\G)}\lambda_\phi(X).$$
\end{defn}

\begin{rem}
	Note that the displacement function - $\lambda_{\phi}$ - is invariant under change of volume, so one can work interchangeably between $\O(\G)$ and $\O_1(\G)$. 
\end{rem}



\begin{thm}[\cite{BestvinaBers,FM13,FM18I}]
	Given $[\phi] \in \Out(\G)$ we define, 
	$$\Min(\phi) = \{ X \in \O(\G) \ : \ \lambda_{\phi}(X) = \lambda(
	\phi)\}.$$
	
	(Similarly for $\O_1(\G)$.) Then, if $\phi$ is irreducible, $\Min(\phi)$ is non-empty and coincides with the set of points which admit a train track map representing $\phi$.
\end{thm}

\begin{rem} There is also a generalisation of the previous theorem for reducible automorphisms,
  but in that case $\Min(\phi)$ may be empty in $\O(\G)$. In any case $\Min(\phi)$ is never
  emtpy if we add to $\O(\G)$ the simplicial bordification at infinity.

  While we don't use the following in this paper, it seems worthwhile mentioning that $\Min(\phi)$
 coincides with the set of points supporting partial train-tracks (which reduce to classical
 train-tracks in irreducible case). (See~\cite{FM13,FM18I} for more details).
\end{rem}

\begin{lem}\label{wasfootnote}
  Let $[\phi]\in\Out(\G)$ be an irreducible element and let $X$ be a minimally displaced
  point. Let $f:X\to X$ be an optimal map representing $\phi$, then
  \begin{enumerate}
  \item The tension graph of $f$ is the whole $X$.
  \item If $f$ is train track then it is a minimal optimal map.
  \end{enumerate}
\end{lem}
\begin{proof}
By \cite[Lemma~4.16]{FM18I} (see also\cite{FM13}), since $f$ is an optimal map representing
  $\phi$, its tension graph contains an invariant sub-graph
  $\phi$. By irreducibility, that sub-graph must be the whole $X$. In~\cite{FM18I} (or also in~\cite{FM13}) it is proved
  that if $f$ is a train track, and
  it is not minimal, then any neighborhood of $X$ in the $\Min(\phi)$ supports an optimal map
  representing $\phi$ whose tension graph is  not the whole $X$, contradicting point $(1)$.
\end{proof}

\providecommand{\bysame}{\leavevmode\hbox to3em{\hrulefill}\thinspace}
\providecommand{\MR}{\relax\ifhmode\unskip\space\fi MR }
\providecommand{\MRhref}[2]{%
  \href{http://www.ams.org/mathscinet-getitem?mr=#1}{#2}
}
\providecommand{\href}[2]{#2}

\end{document}